\documentclass[11pt]{amsart}
\usepackage{tikz}
\usepackage{amsmath,amssymb,amsfonts,enumerate}
\usepackage{upgreek}

\newtheorem{theorem}{Theorem}[section]
          \newtheorem{corollary}[theorem]{Corollary}
          \newtheorem{prop}[theorem]{Proposition}
          \newtheorem{lemma}[theorem]{Lemma}

\newcommand{\mtwo}[4]{\left[\begin{array}{cc}#1&#2\\#3&#4\end{array}\right]}

\newcommand{\rQ}[2]{#1\big\slash#2}
\newcommand{\lQ}[2]{#1\big\backslash#2}
\newcommand{\dQ}[3]{#1\big\backslash#2\big\slash#3}
\newcommand{\biinv}[2]{#1/\!\!/#2}

\newcommand{\op}[1]{\operatorname{#1}}
\newcommand{\ce}[1]{\widetilde{#1}}
\newcommand{\mf}[1]{\mathfrak{#1}}

\newcommand{\tp}[1]{{^\intercal\:\!\!#1}}
\newcommand{\tpinv}[1]{{^\intercal\:\!\!#1^{-1}}}

\newcommand{\GL}{\operatorname{GL}}
\newcommand{\SL}{\operatorname{SL}}

\newcommand{\Sp}{\operatorname{Sp}}
\newcommand{\SO}{\operatorname{SO}}

\newcommand{\ind}{\operatorname{ind}}

\newcommand{\vol}{\operatorname{vol}}
\newcommand{\tr}{\operatorname{tr}}

\newcommand{\Q}{\mathbb{Q}}
\newcommand{\Z}{\mathbb{Z}}
\newcommand{\C}{\mathbb{C}}
\newcommand{\R}{\mathbb{R}}
\newcommand{\F}{\mathbb{F}}

\newcommand{\h}{\mathfrak{h}}
\newcommand{\g}{\mathfrak{g}}

\newcommand{\X}{\mathcal{X}}
\newcommand{\Y}{\mathcal{Y}}
\newcommand{\W}{\mathcal{W}}
\newcommand{\HH}{\mathcal{H}}
\newcommand{\U}{\mathcal{U}}

\newcommand{\inj}{\hookrightarrow}
\newcommand{\proj}{\twoheadrightarrow}

\newcommand{\cf}{cf. }

\title{A minimal even type of the 2-adic Weil representation}
\author{Aaron Wood}
\thanks{Partially supported by NSF grant DMS-0852429.}

\begin{document}
\maketitle

\begin{abstract}
The Weil representation is used to construct a minimal type of a central extension of $\Sp_{2n}(\Q_2)$.  The corresponding Hecke algebra is shown to be isomorphic to the classical affine Hecke algebra of the split adjoint group $\SO_{2n+1}(\Q_2)$.
\end{abstract}

\section*{Introduction}

Let $\W$ be a nondegenerate symplectic space over a $p$-adic field, $G=\ce{\Sp}(\W)$ a two-fold central extension of the symplectic group $\underline{G}=\Sp(\W)$ of type C$_n$, and  $\upomega$ the Weil representation of $G$.  

If $p$ is odd, $\upomega$ has a one-dimensional subspace upon which the inverse image of the Iwahori subgroup acts.  This minimal type was used by Gan and Savin in \cite{gan-savin} to establish an equivalence of categories between certain representations of $G$ and certain other representations of split adjoint groups of type B$_n$.  In this equivalence, the even Weil representation of $G$ corresponds to the trivial representation of the orthogonal group.  The correspondence is realized explicitly as an isomorphism between the Hecke algebra of this type and the classical affine Hecke algebra of type B$_n$.  If $p=2$, there are no Iwahori-fixed vectors. 

In \cite{loke-savin}, Loke and Savin consider the representations of the two-fold central extension $\ce\SL_2(\Q_2)$ which are generated by vectors on which $K$ acts by a character $\upchi$, where $K$ is the full inverse image of a certain congruence subgroup of $\SL_2(\Z_2)$.  The Hecke algebra which captures the structure of such representations is computed explicitly and is shown to be isomorphic to the classical affine Hecke algebra of $\op{PGL}_2(\Q_2)$.

The purpose of this paper is to frame the preceding result for $\ce\SL_2(\Q_2)$ in the language of the Weil representation and to extend it to larger symplectic groups.  This is accomplished by finding a minimal type of the Weil representation and computing the corresponding Hecke algebra.  In the end, the result is the same as in the case of $p$ odd; that is, the Hecke algebra is shown to be isomorphic to that of the split adjoint orthogonal group.

The description of the appropriate open compact subgroup stems from the following fact: in characteristic 2, the bilinear form attached to a symmetric quadratic form is also alternating.  Hence, the finite split orthogonal group $\op{O}_{2n}(\F_2)$ may be realized as a subgroup of the finite symplectic group $\Sp_{2n}(\F_2)$.  If $B'$ is a Borel subgroup of $\op{O}_{2n}(\F_2)$ which sits in a Borel subgroup $B$ of $\Sp_{2n}(\F_2)$, then the inverse image $\underline{J}$ of $B'$  (under the projection map $\Z_2\to\F_2$) is a subgroup of the Iwahori subgroup, as pictured in the diagram below.
\begin{center}
\begin{tikzpicture}[scale=.4]
	\node at (0,4) {$\underline{J}$};
	\node at (4,4) {$\underline{I}$};
	\node at (9.5,4) {$\Sp_{2n}(\Z_2)$};
	\node at (0,0) {$B'$};
	\node at (4,0) {$B$};
	\node at (9.5,0) {$\Sp_{2n}(\F_2)$};
	\draw[->] (1,4) --(3,4);
	\draw[->] (5,4) --(7,4);
	\draw[->] (1,0) --(3,0);
	\draw[->] (5,0) --(7,0);
	\draw[->] (0,3) --(0,1);
	\draw[->] (4,3) --(4,1);
	\draw[->] (9.5,3) --(9.5,1);
\end{tikzpicture}
\end{center}

Let $J$ be the full inverse image in $G$ of $\underline{J}$.  In the $\ce\SL_2$ case, this $J$ is conjugate to the $K$ from \cite{loke-savin}.

Section 1 of this paper is dedicated to fixing notation and summarizing previously known, but relevant, results.  

In Section 2, the $\ce\SL_2$ case is worked out in detail.  Specifically, the line on which $J$ acts is spanned by the characteristic function of $\Z_2$; the support of the associated Hecke algebra $\HH$ is reduced to the set of $J$-double cosets parametrized by the affine Weyl group; and generators $T_0$ and $T_1$ for $\HH$ are given which satisfy only the quadratic relations $(T_0-2)(T_0+1)=0$ and $T_1^2=1$, giving the identification with the affine Hecke algebra of $\op{PGL}_2(\Q_2)$.

In Section 3, the general symplectic case is described.   The line on which $J$ acts is spanned by the characteristic function of the standard lattice; the support of $\HH$ is again reduced to the set of $J$-double cosets parametrized by the affine Weyl group; and generators $T_0,\dots,T_n$ for $\HH$ are given which satisfy the quadratic relations
$$(T_i-2)(T_i+1)=0\;\;\text{ and }\;\;T_n^2=1$$
and the braid relations of type $C_n$.
\begin{center}
\begin{tikzpicture}[scale=.5]
	\draw (0,-.1) --(2,-.1);
	\draw (0,.1) --(2,.1);
	\draw (2,0) --(5,0);
	\draw (7,0) --(8,0);
	\draw (8,-.1) --(10,-.1);
	\draw (8,.1) --(10,.1);
	\draw[fill] (0,0) circle(5pt);
	\draw[fill] (2,0) circle(5pt);
	\draw[fill] (4,0) circle(5pt);
	\draw[fill] (8,0) circle(5pt);
	\draw[fill] (10,0) circle(5pt);
	\node at (0,-.8) {$T_0$};
	\node at (2,-.8) {$T_1$};
	\node at (4,-.8) {$T_2$};
	\node at (10,-.8) {$T_n$};
	\node at (6,0) {$\cdots$};
\end{tikzpicture}
\end{center}

The affine Hecke algebra $\HH'$ of the split adjoint orthogonal group has generators $t_0,\dots,t_n$ and $\uptau$, where $\uptau$ corresponds to the involution of the extended Dynkin diagram of type $B_n$, with $\uptau^2=1$ and $\uptau t_1\uptau=t_0$.
\begin{center}
\begin{tikzpicture}[scale=.5]
	\draw (0,-1.3) --(2,0);
	\draw (0,1.3) --(2,0);
	\draw (2,0) --(3,0);
	\draw (5,0) --(6,0);
	\draw (6,-.1) --(8,-.1);
	\draw (6,.1) --(8,.1);
	\draw[<->] (0,.9) --(0,-.9);
	\draw[fill] (0,1.3) circle(5pt);
	\draw[fill] (0,-1.3) circle(5pt);
	\draw[fill] (2,0) circle(5pt);
	\draw[fill] (6,0) circle(5pt);
	\draw[fill] (8,0) circle(5pt);
	\node at (-.7,1.3) {$t_0$};
	\node at (-.7,-1.3) {$t_1$};
	\node at (2,-.8) {$t_2$};
	\node at (8,-.8) {$t_n$};
	\node at (4,0) {$\cdots$};
	\node at (.4,0) {$\uptau$};
\end{tikzpicture}
\end{center}

The quadratic relations for the generators $t_i$ are $(t_i-2)(t_i+1)=0$.  Since $t_0$ may be expressed in terms of $t_1$ and $\uptau$, it is not needed to define $\HH'$ as an abstract algebra.   The braid relation between $\uptau$ and $t_1$ is $\uptau t_1\uptau t_1=t_1\uptau t_1\uptau$, so the isomorphism from $\HH'$ to $\HH$ is given by
$$t_n\mapsto T_0,\;\;\dots\;\;
t_1\mapsto T_{n-1}\;\;\text{ and }\;\;
\uptau\mapsto T_n.$$

It is shown that this map preserves the trace and $\ast$ operations for Hilbert algebras, which ensures, by \cite{bushnell-henniart-kutzko}, that the induced Plancherel measures coincide in the correspondence of representations given by the Hecke algebra isomorphism.

A portion of this paper is a part of the author's Ph.D. thesis.  The author would like to thank his advisor Gordan Savin for his invaluable guidance throughout this endeavor.

\section{Preliminaries}

In this section, notation will be fixed and some results will be summarized.

\subsection{Symplectic Vector Spaces}

Let $\W$ be a $2n$-dimensional vector space over a field $F$, and let $Q$ be a nondegenerate, skew-symmetric form on $\W$.  Such a form $Q$ is called a \emph{symplectic form} and such a vector space $\W$ is called a \emph{symplectic space}.  

A subspace $\X$ of $\W$ is \emph{isotropic} if $Q$ is identically zero on $\X$.  For a maximal isotropic subspace $\X$ of $\W$, there is a complementary subspace $\Y$ which is also a maximal isotropic subspace of $\W$; each of $\X,\Y$ is a vector subspace of dimension $n$.  Such a decomposition $\W=\X+\Y$ into maximal isotropic subspaces is called a \emph{complete polarization} of $\W$.

Let $\X+\Y$ be a complete polarization of $\W$ and let $\{e_i\}$ be a basis of $\X$.  There is a basis $\{f_i\}$ of $\Y$ such that $Q(e_i,f_j)=\updelta_{ij}$.  The resulting basis $\{e_1,\dots,e_n,f_1,\dots,f_n\}$ of $\W$ is called a \emph{symplectic basis}.  Under such a basis, elements of $\W$ may be considered as column vectors and the symplectic form $Q$ may be expressed as
$$Q(u,v)=\tp{u} \mtwo{0}{1}{-1}{0} v,$$
where $\tp{u}$ is the transpose of $u$.

\subsection{Symplectic Lie Algebras}\label{liealgebra}

Let $\W$ be a $2n$-dimensional symplectic space over $\C$ with symplectic form $Q$.  The \emph{symplectic Lie algebra} $\mf{sp}(\W)$ is defined to be the Lie algebra of linear endomorphisms $T$ of $\W$ satisfying
$$Q(Tu,v)+Q(u,Tv)=0$$
for all $u,v$ in $\W$, with Lie bracket 
$$[T_1,T_2]=T_1T_2-T_2T_1.$$

Under the symplectic basis, $\mf{sp}(\W)$ becomes the Lie subalgebra of $\mf{gl}_{2n}(\C)$ given by
$$\mf{sp}(\W)=\left\{\mtwo{a}{b}{c}{- \tp{a}}\in\mf{gl}_{2n}(\C): b=\tp{b},\;c=\tp{c}\right\}.$$

\subsubsection{Cartan Decomposition}

Let $\h$ be the Cartan subalgebra consisting of diagonal matrices in $\mf{sp}(\W)$ and let $\h^\ast=\op{Hom}_\C(\h,\C)$ be its linear dual.  The Cartan subalgebra is $n$-dimensional and an arbitrary element $H$ of $\h$ is of the form
$$H=\mtwo{a}{}{}{-a},$$
where $a$ is a diagonal matrix with entries $a_1,\dots,a_n$.  The dual basis of $\h^\ast$ is $\uplambda_1,\dots,\uplambda_n$, defined by $\uplambda_i(H)=a_i$.

The set of \emph{roots} in $\h^\ast$ is
$$\Phi=\{\pm\uplambda_i\pm \uplambda_j: i\neq j\big\}\cup\big\{\pm2\uplambda_i\}.$$
The roots in the first set are called \emph{short roots} and those in the second set are called \emph{long roots}.  

For a root $\upalpha$, the root space $\g_\upalpha$ is the set $\{X\in\mf{sp}(\W):[H,X]=\upalpha(H)X\}$.  To be explicit, $\g_\upalpha=\C X_\upalpha$, where $X_\upalpha$ is defined as follows.
\begin{align*}
\text{ If }\upalpha=\uplambda_i-\uplambda_j,&\text{ then }X_\upalpha=\mtwo{E_{ij}}{0}{0}{-E_{ji}};\\
\text{ if }\upalpha=\uplambda_i+\uplambda_j,&\text{ then }X_\upalpha=\mtwo{0}{E_{ij}+E_{ji}}{0}{0};\\
\text{ if }\upalpha=2\lambda_i,&\text{ then }X_\upalpha=\mtwo{0}{E_{ii}}{0}{0};\\
\text{ if }\upbeta=-\upalpha,&\text{ then }X_\upbeta=\tp{X}_\upalpha.
\end{align*}
The Cartan decomposition of $\mf{sp}(\W)$ is
$$\mf{sp}(\W)=\h+\sum_{\upalpha\in\Phi}\g_\upalpha.$$

\subsubsection{Roots}\label{roots}

Consider the real vector space $\h^\ast_\R$ with basis $\uplambda_1,\dots,\uplambda_n$, which is a Euclidean space under the usual dot product, denoted $(\;,\;)$.  This is the inner product which arises from the Killing form on $\h$, \cf \cite{steinberg}.  For each $\uplambda\in \h^\ast_\R$, define $s_\uplambda:\h^\ast_\R\to\h^\ast_\R$ to be the reflection across the hyperplane 
$$P_\uplambda=\{\upmu\in\h^\ast_\R:(\upmu,\uplambda)=0\}.$$

Define $\upalpha_n=2\uplambda_n$ and, for $i=1,\dots,n-1$, define $\upalpha_i=\uplambda_i-\uplambda_{i+1}$.  Then, 
$$\Pi=\{\upalpha_1,\dots,\upalpha_n\}$$
is a set of \emph{simple roots} in $\Phi$ with corresponding \emph{positive roots}
$$\Phi^+=\{\uplambda_i\pm\uplambda_j:i<j\}\cup\{2\uplambda_i\}.$$

The \emph{Weyl group} $W$ of $\mf{sp}(\W)$ is the finite group of reflections generated by $\{s_\upalpha:\upalpha\in\Phi\}.$  It is a Coxeter group and is generated by the set of simple reflections $\{s_1,\dots,s_n\}$, where $s_i=s_{\upalpha_i}$.    The braid relations for the generators of $W$ are given by the following Coxeter diagram, \cf \cite{humphreys-coxeter}, \cite{iwahori-matsumoto}.
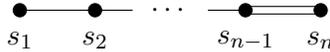
\begin{center}
\begin{tikzpicture}[scale=.5]
	\draw (2,0) --(5,0);
	\draw (7,0) --(8,0);
	\draw (8,-.1) --(10,-.1);
	\draw (8,.1) --(10,.1);
	\draw[fill] (2,0) circle(5pt);
	\draw[fill] (4,0) circle(5pt);
	\draw[fill] (8,0) circle(5pt);
	\draw[fill] (10,0) circle(5pt);
	\node at (2,-.8) {$s_1$};
	\node at (4,-.8) {$s_2$};
	\node at (8,-.8) {$s_{n-1}$};
	\node at (10,-.8) {$s_n$};
	\node at (6,0) {$\cdots$};
\end{tikzpicture}
\end{center}

\subsubsection{Affine Roots}\label{affineroots}

The set of \emph{affine roots} is
$$\Phi^\text{aff}=\{\upalpha+m:\upalpha\in\Phi,m\in\Z\},$$
where an affine root $\upalpha+m$ acts on $\h$ by 
$$(\upalpha+m)(H)=\upalpha(H)+m.$$  
Let $\upalpha_\ast$ be the highest root in $\Phi$,
$$\upalpha_\ast=2\uplambda_1=2\upalpha_1+\dots+2\upalpha_{n-1}+\upalpha_n,$$ 
and define the affine root $\upalpha_0$ by $\upalpha_0=1-\upalpha_\ast.$  Then, 
$$\Pi^\text{aff}=\{\upalpha_0,\upalpha_1,\dots,\upalpha_n\}$$
is a set of \emph{simple affine roots}.

For each affine root $\upgamma=\upalpha+m$, define $s_\upgamma:\h^\ast_\R\to\h^\ast_\R$ to be the reflection across the affine hyperplane 
$$P_\upgamma=P_{\upalpha+m}=\{\upmu\in\h^\ast_\R:(\upmu,\upalpha)=m\}.$$
Let $s_0$ be the affine reflection $s_0=s_{\upalpha_0}.$  

The group  $W^\text{aff}$, generated by these affine reflections, is called the \emph{affine Weyl group}; it is a Coxeter group generated by the simple affine reflections $\{s_0,s_1,\dots,s_n\}$.  The braid relations for these generators are given by the following Coxeter diagram, \cf \cite{humphreys-coxeter}, \cite{iwahori-matsumoto}.
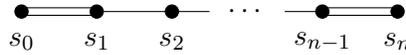
\begin{center}
\begin{tikzpicture}[scale=.5]
	\draw (0,-.1) --(2,-.1);
	\draw (0,.1) --(2,.1);
	\draw (2,0) --(5,0);
	\draw (7,0) --(8,0);
	\draw (8,-.1) --(10,-.1);
	\draw (8,.1) --(10,.1);
	\draw[fill] (0,0) circle(5pt);
	\draw[fill] (2,0) circle(5pt);
	\draw[fill] (4,0) circle(5pt);
	\draw[fill] (8,0) circle(5pt);
	\draw[fill] (10,0) circle(5pt);
	\node at (0,-.8) {$s_0$};
	\node at (2,-.8) {$s_1$};
	\node at (4,-.8) {$s_2$};
	\node at (8,-.8) {$s_{n-1}$};
	\node at (10,-.8) {$s_n$};
	\node at (6,0) {$\cdots$};
\end{tikzpicture}
\end{center}

\subsection{Symplectic Groups}

Let $\W$ be a symplectic space of dimension $2n$ over a field $F$ with symplectic form $Q$.  The \emph{symplectic group} $\Sp(\W)$ is the group of linear automorphisms of $\W$ that preserve the symplectic form, i.e., those linear operators $T:\W\to \W$ such that for all $u,v$ in $\W$,
$$Q(Tu,Tv)=Q(u,v).$$

Under a symplectic basis, $\Sp(\W)$ becomes the matrix subgroup of $\GL_{2n}(F)$ given by
$$\Sp(\W)=\left\{\mtwo{a}{b}{c}{d}\in\GL_{2n}(F):
\begin{array}{l}
\tp{a}d-\!\tp{c}b=1\\
\tp{a}c=\!\tp{c}a\\
\tp{b}d=\!\tp{d}b
\end{array}
\right\}.$$
Note that if $n=1$, then $\Sp(\W)=\SL_2(F)$.
\subsubsection{Symplectic Chevalley Groups}\label{chevalleygroup}

The classical symplectic group over $F$ will be constructed as a Chevalley group from the symplectic Lie algebra over $\C$, \cf \cite{steinberg}, \cite{carter}.  

Let $\W$ be a symplectic vector space over $\C$, let $\U$ be the universal enveloping algebra of $\mf{sp}(\W)$, and let $\U_\Z$ be the subalgebra of $\U$ generated by the divided powers $X_\upalpha^m\big/m!$.  Under the natural representation of $\mf{sp}(\W)$, and hence of $\U$, on $\W$, the elements of $\U$ may be viewed as members of a matrix algebra.  In this setting, each $X_\upalpha^2=0$, so $\U_\Z$ is generated by 1 and $\{X_\upalpha:\upalpha\in\Phi\}$.   Therefore, since $\W$ is the natural representation, the lattice 
$$L=\Z e_1\oplus\dots\oplus\Z e_n\oplus\Z f_1\oplus\dots\oplus\Z f_n$$
is a $\U_\Z$-invariant lattice.

Let $F$ be any field.  For an element $t$ of $F$ and a root $\upalpha$ in $\Phi$, one obtains the natural action of
$$x_\upalpha(t)=\exp(tX_\upalpha)=\sum_{n=0}^\infty t^n\frac{X_\upalpha^n}{n!}=1+tX_\upalpha$$
on $\W_F=L\otimes_\Z F$, which is the $F$-span of $e_1,\dots,e_n,f_1,\dots,f_n$.
In other words, the element $x_\upalpha(t)$ may be interpreted as the actual matrix $1+tX_\upalpha$ with entries in $F$.

Define the \emph{Chevalley group} $G$ to be the group generated by
$$\{x_\upalpha(t): t\in F,\upalpha\in\Phi\}.$$
This group is exactly the classical symplectic group over $F$.

The following elements play an important role in the theory of Chevalley groups.  For $t\in F^\times$, define
\begin{align*}
w_\upalpha(t)&=x_\upalpha(t)x_{-\upalpha}(-1/t)x_\upalpha(t),\\
h_\upalpha(t)&=w_\upalpha(t)w_\upalpha(-1).
\end{align*}
For $i=1,\dots,n$, let $w_i=w_{\upalpha_i}(1).$  The elements $w_1,\dots,w_n$ form a set of representatives of the generators of the Weyl group $W$ in $G$.

Suppose now that $F$ is a nonarchimedian local field with uniformizer $\varpi$ and ring of integers $\mf{O}$.  For any affine root $\upgamma=\upalpha+m$, define the elements
\begin{align*}
x_\upgamma(t)&=x_\upalpha(\varpi^m t),\\
w_\upgamma(t)&=w_\upalpha(\varpi^m t).
\end{align*}  
In addition, define the affine root group
$$\mf{X}_\upgamma=\{x_\upgamma(t):t\in\mf{O}\}.$$
 
For the simple affine root $\upalpha_0=1-\upalpha_\ast$, let 
$$w_0=w_{\upalpha_0}(1)=w_{\upalpha_\ast}(\varpi^{-1}).$$
The elements $w_0,w_1,\dots,w_n$ form a set of representatives in $G$ of the generators of the affine Weyl group, \cf \cite{iwahori-matsumoto}.  

As an abuse of notation, the element $s_\upgamma$ of $W^\text{aff}$ will be identified with its representative $w_\upgamma(1)$ in the Chevalley group.  In particular, these $w_\upgamma(1)$ will often be referred to as elements of the affine Weyl group.

\subsubsection{Generators and Relations}\label{relations}

In \cite{steinberg}, Steinberg gives relations for the symplectic Chevalley group generated by the $x_\upalpha(t)$.  Of particular interest, for any $\upalpha,\upbeta\in\Phi$ and $t,u\in F$ (nonzero where necessary),
\begin{enumerate}[ \mbox{ } (R1)]
\item $x_\upalpha(t)x_\upalpha(u)=x_\upalpha(t+u)$;
\item $\big(x_\upalpha(t),x_\upbeta(u)\big)=\prod x_{i\upalpha+j\upbeta}(c_{ij}t^iu^j)$, for $\upalpha+\upbeta\neq0$;
\item $w_\upalpha(t)x_\upalpha(u)w_\upalpha(-t)=x_{-\upalpha}(-t^{-2}u)$;
\item $h_\upalpha(t)x_\upalpha(u)h_\upalpha(-t)=x_{\upalpha}(t^2 u)$
\item $h_\upalpha(t)h_\upalpha(u)=h_\upalpha(tu)$.
\end{enumerate}
In (R2), the product is taken over the roots that are positive linear combinations of $\upalpha$ and $\upbeta$, and the $c_{ij}$ are constants which will not play a role in this paper.

Besides (R5), the relations are a consequence of (R1) and (R2) if $G=\Sp_{2n}(F)$ for $n\geq2$ or of (R1) and (R3) if $G=\SL_2(F)$.  The two appropriate relations along with (R5) form a complete set of relations for $\Sp_{2n}(F)$.

\subsubsection{Central Extensions}

A group $G'$ is a \emph{central extension} of $G$ if there is a surjective homomorphism from $G'$ to $G$ whose kernel lies in the center of $G'$.  A central extension $E$ of $G$ is \emph{universal} if it a central extension of any other central extension of $G$.  A universal central extension $E$ of the symplectic Chevalley group $G$ exists and is the group defined abstractly using only relations (R1) and (R2).  (In the $\SL_2(F)$ case, the universal central extension is abstractly defined using only relations (R1) and (R3).)  Hence,  the relations (R1) through (R4) can be lifted to any central extension of $G$, \cf \cite{steinberg}.

The preceding paragraph implies that the elements $x_\upalpha(t)$ lift \emph{uniquely} to elements $x_\upalpha'(t)$ in any central extension $G'$.  Therefore, $w_\upalpha(t)$ and $h_\upalpha(t)$ lift \emph{canonically} to $w_\upalpha'(t)$ and $h_\upalpha'(t)$ in $G'$ via the formulas
\begin{align*}
w_\upalpha'(t)&=x_\upalpha'(t)x_{-\upalpha}'(-1/t)x_\upalpha'(t),\\
h_\upalpha'(t)&=w_\upalpha'(t)w_\upalpha'(-1).
\end{align*}
Note that (R5) does not hold in a nontrivial central extension $G'$.

\subsection{Induced Representations and Hecke Algebras}\label{heckealgebras}

Let $H$ be an open compact subgroup of a locally compact, totally disconnected group $G$, with the Haar measure on $G$ normalized to give $\vol(H)=1$.  Denote by $C(G)$ the space of complex-valued functions on $G$ and by $C_0(G)$ the subspace of compactly supported, locally constant functions.  Fix a complex representation $(\uppi,V)$ of $G$ and a character $\upchi$ of $H$.  The objects of interest are
\begin{enumerate}[ \mbox{ } 1.]
\item the (compactly) induced representation $(\upsigma,U)=\ind_H^G\upchi$, where $\upsigma$ acts by right translation on the vector space
$$U=\left\{\upphi\in C(G):
	\begin{tabular}{l}
		$\upphi(hx)=\upchi(h)\upphi(x),h\in H,x\in G$\\
		$\upphi\text{ has compact support modulo }H$
	\end{tabular}\right\};$$
\item the subspace $V^{H,\bar{\upchi}}$ of $V$ given by
$$V^{H,\bar{\upchi}}=\big\{v\in V:\uppi(h)v=\bar{\upchi}(h)v,h\in H\big\};$$
\item the $\upchi$-spherical Hecke algebra $\HH=\HH(\biinv{G}{H};\upchi)$ defined by
$$\HH=\big\{f\in C_0(G): f(h_1xh_2)=\upchi(h_1)f(x)\upchi(h_2),h_i\in H,x\in G\big\},$$
with convolution
$$(f\cdot g)(y)=\int_G f(x)g(x^{-1}y)dx,$$
and identity element $\dot\upchi$, which is the extension of $\upchi$ by 0 to all of $G$.
\end{enumerate}
In this setting, Frobenius Reciprocity gives 
$$\op{Hom}_H(\upchi,\uppi)\cong\op{Hom}_G(\upsigma,\uppi),$$
and $V^{H,\bar{\upchi}}$ is an $\HH$-module under the action
$$\uppi(f)v=\int_G f(x)\big(\uppi(x)v\big)dx.$$

\begin{prop}\label{endomorphism}
$\HH$ is isomorphic to $\op{End}(U)$, and the elements of $\HH$ which are not supported on $H$ act on $U$ as trace-zero endomorphisms.
\end{prop}

\begin{proof}
The maps for this isomorphism are given, but the details are omitted.  The map $\HH\to\op{End}(U)$ is given by $f\mapsto T$ with $T(\upphi)=f\cdot\upphi$.  The inverse map $\op{End}(U)\to\HH$ is given by $T\mapsto f$ with $f(x)=T(\dot{\upchi})(x)$.

Suppose that $f\in\HH$ is not supported on $H$ and let $\upphi\in U$ be supported on a coset $Hx$.  In this setting, $f(y)\upphi(y^{-1}x)=0$ for all $y\in G$, so
$$(f\cdot\upphi)(x)=\int_Gf(y)\upphi(y^{-1}x)dy=0.$$
The functions supported on cosets $Hx$, for $x\in\lQ{H}{G}$, form a basis of $U$, so this calculation implies that the trace of $f$ is zero.
\end{proof}

There is additional structure on $\HH$, namely the $\ast$-operation given by
$$f^\ast(x)=\overline{f(x^{-1})}$$
and the trace given by
$$\tr(f)=f(1).$$
Following \cite{bushnell-henniart-kutzko}, $\HH$ is a normalized Hilbert algebra with involution $f\mapsto f^\ast$ and scalar product $[\;,\,]$, given by
$$[f,g]=\tr(f^\ast\cdot g).$$
This structure yields a Plancherel formula on $\HH$, meaning that $\tr(f)$ may be defined in terms of the traces $\tr\uprho(f)$, for irreducible representations $\uprho$ of $\HH$.  For details, see 3.2 of \cite{bushnell-henniart-kutzko}.

If $\uppi$ is unitary, under the identification $\HH=\op{End}(U)$,  $\ast$ is the adjoint operator and $\tr$ is the usual trace map.  By definition,
$$\tr(f^\ast\cdot f)=\int_G f^\ast(x)f(x^{-1})dx=||f||^2,$$
so the $L^2$ norm may be expressed in terms of $\ast$ and $\tr$.

For the remainder of this section, assume that $G$ is compact.

\begin{prop}\label{T_x}
Suppose that $V^{H,\bar\upchi}\neq0$ and that $\HH$ is supported at most on $H$ and $HxH$, for $x\notin H$.  If $\dim(V)<\dim(U)$, then there exists nontrivial $T_x\in\HH$ supported on $HxH$ which acts on the eigenspace $V$ by
$$\uplambda=\frac{\dim(U)-\dim(V)}{\dim(V)};$$
that is, $T_x$ satisfies the quadratic relation 
$$(T_x-\uplambda)(T_x+1)=0.$$
\end{prop}

\begin{proof}
Assume the hypothesis and let $(\uppi^\ast,V^\ast)$ be the dual representation of $(\uppi,V)$.  Since $V^\ast$ has smaller dimension than $U$, it may be realized, via Frobenius reciprocity, as a quotient of $U$; hence, $U$ is reducible.
Therefore, $\HH=\op{End}(U)$ is exactly 2-dimensional, so there exists a nontrivial element $T_x$ supported on $HxH$.  Since $T_x$ is not supported on $H$, it acts as a trace-zero endomorphism on $U$ (by Proposition \ref{endomorphism}) with two eigenspaces $V$ and $V^\perp$.  Let $\uplambda$ and $\upmu$ be their respective eigenvalues so that 
$$\uplambda\dim(V)+\upmu\dim(V^\perp)=0.$$
Normalize $T_x$ to act by $\upmu=-1$ on $V^\perp$; then, $T_x$ acts on $V$ by
$$\uplambda=\frac{\dim(V^\perp)}{\dim(V)},$$
hence the proposition.
\end{proof}

As a remark, the $\HH$-modules $U$ and $V^{H,\bar\upchi}$ are compatible in the following way.  Let $V_1,\dots,V_n$ be the irreducible representations of $G$ such that $V_i^{H,\bar\upchi}$ is nontrivial.  The matrix coefficients $v_i^\ast(\uppi_i(g^{-1})v_i)$, for $v_i^\ast\in V_i^\ast$ and $v_i\in V_i$, give a part $\bigoplus V_i\otimes V_i^\ast$ of the decomposition of $C(G)$, and hence,
$$U^{H,\bar\upchi}=\bigoplus_{i=1}^n V_i^{H,\bar\upchi}\otimes V_i^\ast.$$
In particular, if $\upphi(g)=v_i^\ast(\uppi_i(g^{-1})v_i)$, then
\begin{align*}
(f\cdot\upphi)(x)&=\int_G f(g)v_i^\ast\big(\uppi_i(x^{-1}g)v_i\big)dg\\
&=v_i^\ast\big(\uppi_i(x^{-1})(\uppi_i(f)v_i)\big)\\
&=\uppi_i(f)(v_i\otimes v_i^\ast).
\end{align*}

\subsection{Additive Characters of $\Q_p$}

For an additive character of $\uppsi$ of $\Q_p$, if there exists a smallest integer $c$ such that $\uppsi$ is trivial on $p^c\Z_p$, then $c$ is called the \emph{conductor of} $\uppsi$.

There is a natural projection from $\Q_p$ to $\rQ{\Q_p}{\Z_p}$, a canonical embedding of $\rQ{\Q_p}{\Z_p}$ into the $p$-torsion of $\rQ{\Q}{\Z}$, and an embedding of $\rQ{\Q}{\Z}$ into $\C^\times$ given by $x\mapsto e^{2\pi ix}$.  The composition of these maps 
$$\Q_p\proj\rQ{\Q_p}{\Z_p}\inj\rQ{\Q}{\Z}\inj\C^\times$$
defines an additive character $\uppsi_1:\Q_p\to\C^\times$ of conductor $c=0$.

For $a$ in $\Q_p^\times$, define the character $\uppsi_a:\Q_p\to\C^\times$ by
$$\uppsi_a(x)=\uppsi_1(ax),$$
which has conductor $c=-\op{val}(a)$.  These $\uppsi_a$ account for all nontrivial smooth additive characters of $\Q_p$, \cf \cite{washington}.

\subsection{Fourier Transform}\label{fouriertransform}

Let $S(\Q_p)$ denote the set of Schwarz functions on $\Q_p$, i.e., the set of smooth, compactly supported, complex-valued functions on $\Q_p$. Fix an additive character $\uppsi$ of $\Q_p$ of conductor $c$.  The Fourier transform (with respect to $\uppsi$) on $S(\Q_p)$ is given by $f\mapsto\widehat{f}$ where
$$\widehat{f}(y)=\int_{\Q_p}\uppsi(2uy)f(u)du.$$
The Haar measure on $\Q_p$ will be normalized to give $\widehat{\widehat{f}\;}(y)=f(-y).$  The appearance of a 2 in the Fourier transform affects the role of the conductor of $\uppsi$ for $p=2$, so it will be convenient to define
$$\updelta=
\begin{cases}
1&\text{ if }p=2\\
0&\text{ if }p\neq2.
\end{cases}$$
For $m\in\Z$, denote the characteristic function of $p^m\Z_p$ by $\upphi_m$.  The proofs of following statements are straightforward computations.

\begin{enumerate}[ \mbox{ } 1.]
\item If $f$ is supported on $p^m\Z_p$ and constant on $p^n\Z_p$-cosets, then $\widehat{f}$ is supported on $p^{-n+(c-\updelta)}\Z_p$ and constant on $p^{-m+(c-\updelta)}\Z_p$-cosets.
\item The Fourier transform of $\upphi_m$ is $\widehat\upphi_m=\vol(p^m\Z_p)\upphi_{-m+(c-\updelta)}$.
\item The normalized volume of $\Z_p$ is $\vol(\Z_p)=p^{-(c-\updelta)/2}$.
\end{enumerate}
Hence, if $c=\updelta$, then the Fourier transform of $\upphi_m$ is $\widehat{\upphi}_m=p^{-m}\upphi_{-m}$.

Suppose now that $\Y$ is a vector space over $\Q_p$, where the elements of $\Y$ are considered as column vectors with respect to some basis.  Let $S(\Y)$ denote the set of Schwarz functions on $\Y$ and fix an additive character $\uppsi$ of $\Q_p$.  The Fourier transform with respect to $\uppsi$ on $S(\Y)$ is defined by $\mathbf{f}\mapsto\widehat{\mathbf{f}}$ where
$$\widehat{\mathbf{f}}(y)=\int_\Y \uppsi(2\;\tp{u}y)\mathbf{f}(u)du.$$
The Haar measure on $\Y$ is normalized to give $\widehat{\widehat{\mathbf{f}}}(-y)=\mathbf{f}(y)$.

The space $S(\Y)$ has a tensor product structure
$$S(\Y)=S(\Q_p)\otimes\cdots\otimes S(\Q_p),$$
and the Fourier transform is well-behaved in this setting.  In particular, if $\mathbf{f}=f_1\otimes\cdots\otimes f_n$, then 
$\widehat{\mathbf{f}}=\widehat{f}_1\otimes\cdots\otimes\widehat{f}_n.$

\subsection{Weil Representation}

The following exposition of the Weil representation mostly follows \cite{kudla}.  Let $F$ be a nonarchimedian local field, let $\W$ be a $2n$-dimensional symplectic space over $F$ with symplectic form $Q$, and let $\uppsi$ be an additive character of $F$.

The Heisenberg group $H(\W)$ is defined to be the set $\W\times F$ with group multiplication
$$(u,s)\cdot(v,t)=\big(\,u+v\,,\;s+t+Q(u,v)\big).$$
The center of $H(\W)$ is $Z=\{(0,t)\}= F$.  Since the symplectic group $\Sp(\W)$ preserves $Q$, it acts as a group of automorphisms on $H(\W)$ by
$$g(v,t)=(gv,t).$$

Let $(\uprho,S)$ be a representation of $H(\W)$ with central character $\uppsi$; that is, $\uprho(0,t)\upphi=\uppsi(t)\upphi$ for any $\upphi$ in $S$.  One can twist $\uprho$ by $g\in\Sp(\W)$ to obtain the representation $(\uprho^g,S)$ given by
$$\uprho^g(v,t)=\uprho\big(g(v,t)\big)=\uprho(gv,t),$$
which also has central character $\uppsi$.   By the Stone-von Neumann theorem, $\uprho^g$ must be isomorphic to $\uprho$, so, for each $g\in\Sp(\W)$, there exists an operator $T(g):S\to S$, unique up to a scalar in $\C^\times$, which intertwines these two representations; that is,
$$T(g)\uprho=\uprho^g T(g).$$
The operator $T:\Sp(\W)\to\GL(S)\big/\C^\times$ is a projective representation of $\Sp(\W)$.  Let $\ce{\Sp}(\W)$ be a two-fold central extension of $\Sp(\W)$, 
$$1\to\{\pm1\}\to\ce{\Sp}(\W)\to\Sp(\W)\to1.$$
It is known that $T$ lifts uniquely to a linear representation 
$$\upomega:\ce{\Sp}(\W)\to\GL(S),$$ 
called the \emph{Weil representation} with respect to $\uppsi$, \cf \cite{kudla}, \cite{weil}.

\subsubsection{Models}

For any closed subgroup $\mathcal{Z}$ of $\W$ define
$$\mathcal{Z}^\perp=\big\{v\in \W:\uppsi\big(Q(v,z)\big)=1\text{ for all }z\in\mathcal{Z}\big\};$$
this is also a closed subgroup of $\W$.   Define $H(\mathcal{Z})$ to be the subgroup $\mathcal{Z}\times F$ of the Heisenberg group, and assume that $\mathcal{Z}\subset\mathcal{Z}^\perp$.  The character $\uppsi$ can be extended trivially to $H(\mathcal{Z})$ by $\uppsi(z,t)=\uppsi(t)$; indeed,
\begin{align*}
\uppsi\big((z_1,t_1)\cdot(z_2,t_2)\big)
&=\uppsi\big(z_1+z_2,t_1+t_2+Q(z_1,z_2)\big)\\
&=\uppsi(t_1)\uppsi(t_2)\uppsi\big(Q(z_1,z_2)\big)\\
&=\uppsi(z_1,t_1)\uppsi(z_2,t_2).
\end{align*}
Define $(\uprho,S_\mathcal{Z})$ to be the induced representation $\ind_{H(\mathcal{Z})}^{H(\W)}\uppsi$ of $H(\W)$; that is, $\uprho$ acts by right translation on the space $S_\mathcal{Z}$ of smooth functions $f$ on $H(\W)$ which satisfy the following properties:
\begin{enumerate}[ \mbox{ } 1.]
\item $f$ has compact support modulo $H(\mathcal{Z})$;
\item $f(zh)=\uppsi(t)f(h)$ for $(z,t)\in H(\mathcal{Z})$ and $h\in H(\W)$;
\item there is an open compact subgroup $K_f$ of $H(\W)$ such that $f(hk)=f(h)$ for $k\in K_f$ and $h\in H(\W)$.
\end{enumerate}
The restriction of $\uprho$ to the center of $H(\W)$ is given by
$$\big(\uprho(0,t)f\big)(h)=f\big(h\cdot(0,t)\big)=f\big((0,t)\cdot h\big)=\uppsi(t)f(h),$$
hence $\uprho$ has central character $\uppsi$.

For different such closed subgroups $\mathcal{Z}$ of $\W$, one can build a model for $\uprho$ and hence a model for $\upomega$.  If $\mathcal{Z}$ is a maximal isotropic subspace of $\W$, the model is called Shr\"odinger's model.

\subsubsection{Schr\"odinger's Model}

Let $\X+\Y$ be a complete polarization of $\W$.  As $\X$ is a maximal isotropic subspace of $\W$, $Q$ is identically zero on $\X$, and hence $\X=\X^\perp$.  Following the construction above, one obtains the space of functions $S_\X$.  This space is canonically isomorphic to the space $S(\Y)$ of Schwarz functions on $\Y$ via the map $S_\X\to S(\Y)$ given by $f\mapsto\upphi$, where 
$$\upphi(y)=f(y,0).$$

\begin{prop}
In Schr\"odinger's model, $\uprho$ takes the form
\begin{align*}
\big(\uprho(x,0)\upphi\big)(y_0)&=\uppsi\big(-2Q(x,y_0)\big)\upphi(y_0)\\
\big(\uprho(y,0)\upphi\big)(y_0)&=\upphi(y+y_0)\\
\big(\uprho(0,t)\upphi\big)(y_0)&=\uppsi(t)\upphi(y_0).
\end{align*}
\end{prop}

\begin{proof}
Let $t\in F$, $x\in\X$, and $y_0,y\in\Y$.  Then,
\begin{align*}
\big(\uprho(x+y,t)\upphi\big)(y_0)
&=f\big((y_0,0)\cdot(x+y,t)\big)\\
&=f\big(x+y+y_0\,,\,t+Q(y_0,x+y)\big)\\
&=f\big((x,t+Q(y_0,x)-Q(x,y+y_0))\cdot(y+y_0,0)\big)\\
&=\uppsi\big(x,t-Q(x,y+2y_0)\big)f(y+y_0,0)\\
&=\uppsi\big(t-Q(x,y+2y_0)\big)\upphi(y+y_0),
\end{align*}
hence the proposition.
\end{proof}

The group $\Sp(\W)$, under the symplectic basis corresponding to the polarization $\X+\Y$, is generated by the following types of matrices:
\begin{align*}
\underline{x}(b)&=\mtwo{1}{b}{}{1}\hskip10pt\text{ ($b$ a symmetric $n\times n$ matrix),}\\
\underline{h}(b)&=\mtwo{b}{}{}{\tpinv{b}}\hskip10pt\text{ ($b$ an invertible $n\times n$ matrix),}\\
\underline{w}&=\mtwo{}{1}{-1}{}.
\end{align*}

\begin{prop}
The intertwining map $T$ is given by
\begin{align*}
T\big(\underline{x}(b)\big)\upphi(y)&=\uppsi(\tp{y}by)\upphi(y)\\
T\big(\underline{h}(b)\big)\upphi(y)&=\upphi(\tp{b}y)\\
T(\underline{w})\upphi(y)&=\widehat{\upphi}(y).
\end{align*}
\end{prop}

\begin{proof}
The action of $(0,t)$ under $\uprho$ will clearly be intertwined by these operators, so it suffices to verify the proposition for $\uprho_0(x,y)=\uprho(x+y,0)$.  First, the formula $T(g)\uprho=\uprho^g T(g)$ is verified for $g=\underline{x}(b)$:
\begin{align*}
\uprho_0^{\underline{x}(b)}(x,y)&\Big(T\big(\underline{x}(b)\big)\upphi\Big)(y_0)\\
&=\uprho_0(x,bx+y)\Big(T\big(\underline{x}(b)\big)\upphi\Big)(y_0)\\
&=\uppsi\big(-\tp{(x+by)}(y+2y_0)\big)\Big(T\big(\underline{x}(b)\big)\upphi\Big)(y+y_0)\\
&=\uppsi\big(-\tp{(x+by)}(y+2y_0)\big)\uppsi\big(\tp{(y+y_0)}b(y+y_0)\big)\upphi(y+y_0)\\
&=\uppsi(\tp{y_0}by_0)\uppsi\big(-\tp{x}(y+2y_0)\big)\upphi(y+y_0)\\
&=\uppsi(\tp{y_0}by_0)\big(\uprho_0(x,y)\upphi\big)(y_0)\\
&=T\big(\underline{x}(b)\big)\big(\uprho_0(x,y)\upphi\big)(y_0).
\end{align*}
Next, it is verified for $g=\underline{h}(b)$:
\begin{align*}
\uprho_0^{\underline{h}(b)}(x,y)&\Big(T\big(\underline{h}(b)\big)\upphi\Big)(y_0)\\
&=\uprho_0(bx,\tpinv{b}y)\Big(T\big(\underline{h}(b)\big)\upphi\Big)(y_0)\\
&=\uppsi\big(-\tp{x}(y+2\:\tp{b}y_0)\big)\Big(T\big(\underline{h}(b)\big)\upphi\Big)(y_0)\\
&=\uppsi\big(-\tp{x}(y+2\:\tp{b}y_0)\big)\upphi(y+\tp{b}y_0)\\
&=\big(\uprho_0(x,y)\upphi\big)(\tp{b}y_0)\\
&=T\big(\underline{h}(b)\big)\big(\uprho_0(x,y)\upphi\big)(y_0).
\end{align*}
Finally, it is verified for $g=\underline{w}$:
\begin{align*}
\uprho_0^w(x,y)&\big(T(\underline{w})\upphi\big)(y_0)\\
&=\uprho_0(y,-x)\widehat{\upphi}(y_0)\\
&=\uppsi\big(-\tp{y}(-x+2y_0)\big)\widehat{\upphi}(-x+y_0)\\
&=\uppsi\big(-\tp{y}(-x+2y_0)\big)\int_\Y\uppsi\big(2\:\tp{u}(-x+y_0)\big)\upphi(u)du\\
&=\uppsi\big(-\tp{y}(-x+2y_0)\big)\int_\Y\uppsi\big(2\:\tp{(v+y)}(-x+y_0)\upphi(v+y)dv\\
&=\int_\Y\uppsi(2\:\tp{v}y_0)\uppsi\big(-\tp{x}(y+2v)\big)\upphi(v+y)dv\\
&=\int_\Y\uppsi(2\:\tp{v}y_0)\big(\uprho_0(x,y)\upphi\big)(v)dv\\
&=T(\underline{w})\big(\uprho_0(x,y)\upphi\big)(y_0).
\end{align*}
\end{proof}

Considering $\Sp(\W)$ as a Chevalley group (Section \ref{chevalleygroup}), the elements $\underline{x}(b)$, $\underline{h}(b)$, and $\underline{w}$ may be expressed as products of the $\underline{x}_\upalpha(t)$.  These $\underline{x}_\upalpha(t)$ lift uniquely to $x_\upalpha(t)$ in $\ce\Sp(\W)$.  In this way, canonical lifts $x(b)$, $h(b)$ and $w$ are defined. The Schr\"odinger model of the Weil representation of $\ce\Sp(\W)$ on $S(\Y)$ is given by
\begin{align*}
x(b)\upphi(y)&=\upalpha_b\uppsi(\tp{y}by)\upphi(y),\\
h(b)\upphi(y)&=\upbeta_b\upphi(\tp{b}y),\\
w\upphi(y)&=\upgamma_1\widehat{\upphi}(y),
\end{align*}
for some constants $\upalpha_b,\upbeta_b,\upgamma_1$.  The computation of these constants is not relevant to the task at hand, so it will suffice to mention that $\upalpha_b=1$, that $\upbeta_b$ is a 4th root of 1, and that $\upgamma_1$ is an 8th root of 1, \cf \cite{kudla}, \cite{weil}.

\section{A Minimal Even Type for $\ce{\SL}_2(\Q_2)$}\label{SL_2}

Throughout the section, $\underline{G}$ is the linear group $\SL_2(\Q_2)$, $G$ is a two-fold central extension of $\underline{G}$, $\uppsi$ is an additive character of $\Q_2$ with conductor 1, $V$ is the space $S(\Q_2)$ of Schwarz functions, $(\upomega,V)$ is the Schr\"odinger model of the Weil representation associated to $\uppsi$, and $\upphi_m$ is the characteristic function of $2^m\Z_2$.  In this setting, $\widehat\upphi_m=2^{-m}\upphi_{-m}$.

As a Chevalley group, $\underline{G}$ is generated by the elements 
$$\underline{x}(t)=\mtwo{1}{t}{}{1}\;\text{ and }\;\;
\underline{y}(t)=\mtwo{1}{}{t}{1},$$
where $t\in\Q_2$.  For $t\in\Q_2^\times,$ the elements $\underline{w}(t)$ and $\underline{h}(t)$ are given by
$$\underline{w}(t)=\underline{x}(t)\underline{y}(-1/t)\underline{x}(t)=\mtwo{}{t}{-t^{-1}}{},$$
$$\underline{h}(t)=\underline{w}(t)\underline{w}(-1)=\mtwo{t}{}{}{t^{-1}}.$$

Let $x(t)$ and $y(t)$ be the unique lifts of $\underline{x}(t)$ and $\underline{y}(t)$; these elements generate the central extension $G$ of $\underline{G}$.  For $t\in\Q_2^\times$, define the elements $w(t)$ and $h(t)$ of $G$ by
$$w(t)=x(t)y(-1/t)x(t),$$
$$h(t)=w(t)w(-1).$$
These are the canonical lifts of $\underline{w}(t)$ and $\underline{h}(t)$, respectively.

The elements $w_0=w(1/2)$ and $w_1=w(1)$ form a set of representatives of the generators $\{s_0,s_1\}$ of the affine Weyl group $W^\text{aff}$.  A full set of representatives of $W^\text{aff}$ in $G$ is $\big\{w(2^n),h(2^n): n\in\Z\big\}$, and, abusing notation, these will often be referred to as elements of the affine Weyl group itself.

The Schr\"odinger model of the Weil representation of $G$ on $V$ is given by
\begin{align*}
x(t)\upphi(y)&=\uppsi(ty^2)\upphi(y)\\
h(t)\upphi(y)&=\upbeta_t\upphi(ty)\\
w_1\upphi(y)&=\upgamma_1\widehat{\upphi}(y).
\end{align*}

\subsection{Minimal Type}

Let $\underline{I}$ be the Iwahori subgroup 
$$\underline{I}=\left\{\mtwo{a}{b}{2c}{d}\in\SL_2(\Z_2):a,b,c,d\in\Z_2\right\}$$
and $\underline{J}$ the normal subgroup of $\underline{I}$ given by
$$\underline{J}=\left\{\mtwo{a}{2b}{2c}{d}\in\SL_2(\Z_2):a,b,c,d\in\Z_2\right\}.$$
Consider the projection map $\uprho:\SL_2(\Z_2)\to\SL_2(\F_2)$.  The image of $\underline{I}$ under $\uprho$ is the Borel subgroup 
$$B_2(\F_2)=\left\{\mtwo{1}{t}{}{1}:t\in\F_2\right\}\cong\F_2,$$
and the image of $\underline{J}$ is the trivial subgroup.  Note that
$$[\SL_2(\Z_2):\underline{J}]=|\SL_2(\F_2)|=6.$$

Let $I$ and $J$ be the full inverse images in the central extension $G$ of $\underline{I}$ and $\underline{J}$, respectively, so that $I$ is generated by
$$\{h(t):t\in\Z_2^\times\}\cup\{x(t):t\in\Z_2\}\cup\{y(t):t\in2\Z_2\}$$
and $J$ is generated by
$$\{h(t):t\in\Z_2^\times\}\cup\{x(t),y(t):t\in2\Z_2\}.$$
From the linear group computations, the quotients $\rQ{I}{J}$ and $\lQ{J}{I}$ are each isomorphic to $B_2(\F_2)$ via
$$x(1)J\longleftrightarrow Jx(1)\longleftrightarrow \mtwo{1}{1}{}{1}.$$

Consider the even function $\upphi_0$ in $V$, and recall that $\widehat\upphi_0=\upphi_0$, so that $w_1\upphi_0\in\C\upphi_0$.  The elements of $G$ that stabilize $\C\upphi_0$ will be determined.
\begin{enumerate}[ \mbox{ } 1.]
\item $x(t)\upphi_0(y)=\uppsi(ty^2)\upphi_0(y)$; the value of $\uppsi(ty^2)$ is independent of $y\in\Z_2$ if and only if $t\in2\Z_2$.  So,
$$x(t)\upphi_0\in\C\upphi_0\text{ if and only if }t\in2\Z_2.$$
\item $h(t)\upphi_0(y)=\upbeta_t\upphi_0(ty)$; the equality $\upphi_0(y)=\upphi_0(ty)$ for all $y\in\Z_2$ is possible if and only if $t\in\Z_2^\times$.  Thus,
$$h(t)\upphi_0\in\C\upphi_0\text{ if and only if }t\in\Z_2^\times.$$
\item Since $w(t)=h(t)w_1$,
$$w(t)\upphi_0\in\C\upphi_0\text{ if and only if }t\in\Z_2^\times.$$
\item Since $y(t)=w_1x(-t)w_1^{-1}$,
$$y(t)\upphi_0\in\C\upphi_0\text{ if and only if }t\in2\Z_2.$$
\end{enumerate}
In particular, the stabilizer of $\C\upphi_0$ is exactly the group $J\cup Jw_1$; that is, $J$ is the intersection of $I$ and the stabilizer of $\C\upphi_0$.

Since $J$ acts on $\C\upphi_0$, denote by $\bar\upchi:J\to\C^\times$ the character satisfying $\upomega(j)\upphi_0=\bar\upchi(j)\upphi_0$ for $j\in J$, which acts on the nontrivial subspace
$$V^{J,\bar\upchi}=\big\{\upphi\in V:\upomega(j)\upphi=\bar\upchi(j)\upphi\text{ for }j\in J\big\}\supset\C\upphi_0.$$

\begin{prop}\label{chibar}
The character $\bar\upchi$ satisfies the following properties.
\begin{enumerate}[ \mbox{ } 1.]
\item $\bar\upchi\big(x(t)\big)=\bar\upchi\big(y(t)\big)=1$ for $t\in2\Z_2$.
\item $\bar\upchi\big(x(-1)y(t)x(1)\big)=\bar\upchi\big(x(1)y(t)x(-1)\big)=\uppsi(-t/4)$ for $t\in2\Z_2$.
\item $V^{J,\bar\upchi}=\C \upphi_0$.
\end{enumerate}
\end{prop}

\begin{proof}
The first part is clear from the determination of $J$.
Noting that 
\begin{align*}
\bar\upchi\big(x(-1)y(t)x(1)\big)
&=\bar\upchi\big(x(2)\big)\bar\upchi\big(x(-1)y(t)x(1)\big)\bar\upchi\big(x(-2)\big)\\
&=\bar\upchi\big(x(1)y(t)x(-1)\big),
\end{align*}
in order to prove the second part, it suffices to compute the action of 
$$x(1)y(t)x(-1)=x(1)w_1x(-t)w_1^{-1}x(-1)$$
on $\upphi_0$.  The action of $x(-1)$ on $\upphi_0$ is given by
\begin{align*}
x(-1)\upphi_0(y)
&=\uppsi(-y^2)\upphi_0(y)\\
&=\left\{\!\!\!\begin{array}{rl}
1&\text{ if }y\in2\Z_2\\
-1&\text{ if }y\in\Z_2\smallsetminus2\Z_2\\
0&\text{ otherwise}
\end{array}\right.\\
&=\big(-\upphi_0+2\upphi_1\big)(y).
\end{align*}
Therefore,
\begin{align*}
x(1)w_1x(-t)w_1^{-1}&\big(x(-1)\upphi_0\big)(y)\\
&=x(1)w_1x(-t)w_1^{-1}\big(-\upphi_0+2\upphi_1\big)(y)\\
&=\upgamma_1^{-1}x(1)w_1\big(\uppsi(-ty^2)(-\upphi_0+\upphi_{-1})\big)(y)\\
&=x(1)\int_{\Q_2}\uppsi(2uy)\uppsi(-tu^2)(-\upphi_0+\upphi_{-1})(u)du\\
&=\uppsi(y^2)\int_{1/2+\Z_2}\uppsi(2uy)\uppsi(-tu^2)du,
\end{align*}
the last equality following from the fact that $-\upphi_0+\upphi_{-1}$ is equal to the characteristic function of $1/2+\Z_2$.  Evaluating this expression at $y=0$ gives
$$x(1)y(t)x(-1)\upphi_0(0)=\int_{1/2+\Z_2}\uppsi(-tu^2)du.$$
If $u\in1/2+\Z_2$, then $u^2\in1/4+\Z_2$, so, for $t\in2\Z_2$, $\uppsi(-tu^2)=\uppsi(-t/4)$.  Since $\vol(\Z_2)=1$, one has
$$x(1)y(t)x(-1)\upphi_0(0)=\uppsi(-t/4).$$

To prove the third part, it is enough to show that $V^{J,\bar\upchi}\subset\C\upphi_0$.  Let $\upphi$ be an arbitrary element of $V^{J,\bar\upchi}$. There exist $m\leq n$ such that $\upphi $ is supported on $2^m\Z_2$ and $\widehat{\upphi}$ is supported on $2^{-n}\Z_2$.  Since $x(2)$ and $y(2)$ act trivially on $V^{J,\bar\upchi}$, it must be that $\uppsi(2y^2)=1$ for all $y\in2^m\Z_2$ and all $y\in2^{-n}\Z_2$.  Therefore, $0\leq m\leq n\leq0$, so $\upphi\in\C\upphi_0$.
\end{proof}

\subsection{Hecke Algebra}

Recalling Section \ref{heckealgebras} and its notation, the Hecke algebra for this type is $\HH=\HH(\biinv{G}{J};\upchi)$.  An element $f$ of $\HH$ is determined by its value on a set of representatives of $\dQ{J}{G}{J}$.  By \cite{iwahori-matsumoto}, $\underline{G}=\underline{I}W^\text{aff}\underline{I}$, so the double cosets $\dQ{I}{G}{I}$ are parametrized by the affine Weyl group.  In this way, a typical representative of a $J$ double coset is of the form $Jx_1wx_2J$, where $x_1$ and $x_2$ are either $x(1)$ or trivial, and $w$ is of the form $w(2^n)$ or $h(2^n)$.  Some of these representatives are redundant and will be eliminated.  

Recalling relations (R3) and (R4) from Section \ref{relations}, one has
$$\begin{array}{rccccl}
Jh(2^n)x(1)J&=&Jx(2^{2n})h(2^n)J&=&Jh(2^n)J&\text{ for }n>0,\\
Jx(1)h(2^n)J&=&Jh(2^n)x(2^{-2n})J&=&Jh(2^n)J&\text{ for }n<0,\\
Jx(1)w(2^n)J&=&Jw(2^n)y(-2^{-2n})J&=&Jw(2^n)J&\text{ for }n<0,\\
Jw(2^n)x(1)J&=&Jy(-2^{-2n})w(2^n)J&=&Jw(2^n)J&\text{ for }n<0,
\end{array}$$
and so,
$$Jx(1)h(2^n)x(1)J=
\begin{cases}
Jx(1)h(2^n)J&\text{ if }n>0,\\
Jh(2^n)x(1)J&\text{ if }n<0,\\
Jh(2^n)J&\text{ if }n=0.
\end{cases}$$
Therefore, a complete set of double coset representatives for $\dQ{J}{G}{J}$ is
$$\begin{array}{rl}
w(2^n)\;\;:&n\in\Z,\\
h(2^n)\;\;:&n\in\Z,\\
w(2^n)x(1)\;\;:&n\geq0,\\
x(1)w(2^n)\;\;:&n\geq0,\\
x(1)w(2^n)x(1)\;\;:&n\geq0,\\
x(1)h(2^n)\;\;:&n\geq0,\\
h(2^n)x(1)\;\;:&n<0.
\end{array}$$

\begin{lemma}\label{SL_2-support-H1}
The support of $\HH$ is contained in $JW^{\emph{aff}}J$.
\end{lemma}

\begin{proof}
Let $f$ be an element of $\HH$.  Recalling relations (R3) and (R4) and Proposition \ref{chibar}, if $n\geq0$,  then
\begin{align*}
f\big(w(2^n)x(1)\big)
&=\upchi\big(x(2^{2n+1})\big)f\big(w(2^n)x(1)\big)\\
&=f\big(x(2^{2n+1})w(2^n)x(1)\big)\\
&=f\big(w(2^n)y(2)x(1)\big)\\
&=f\big(w(2^n)x(1)\big)\upchi\big(x(-1)y(2)x(1)\big)\\
&=\uppsi(-1/2)f\big(w(2^n)x(1)\big),
\end{align*}
\begin{align*}
f\big(x(1)w(2^n)\big)
&=f\big(x(1)w(2^n)x(-2^{2n+1})\big)\\
&=\upchi\big(x(1)y(2)x(-1)\big)f\big(x(1)w(2^n)\big)\\
&=\uppsi(-1/2)f\big(x(1)w(2^n)\big),
\end{align*}
\begin{align*}
f\big(x(1)w(2^n)x(1)\big)
&=f\big(x(1)w(2^n)x(1)x(-2^{2n+1})\big)\\
&=\upchi\big(x(1)y(2)x(-1)\big)f\big(x(1)w(2^n)x(1)\big)\\
&=\uppsi(-1/2)f\big(x(1)w(2^n)x(1)\big),
\end{align*}
\begin{align*}
f\big(x(1)h(2^n)\big)
&=f\big(x(1)h(2^n)y(2^{2n+1})\big)\\
&=\upchi\big(x(1)y(2)x(-1)\big)f\big(x(1)h(2^n)\big)\\
&=\uppsi(-1/2)f\big(x(1)h(2^n)\big),
\end{align*}
and, if $n<0$, then
\begin{align*}
f\big(h(2^n)x(1)\big)
&=f\big(y(2^{-2n+1})h(2^n)x(1)\big)\\
&=f\big(h(2^n)x(1)\big)\upchi\big(x(-1)y(2)x(1)\big)\\
&=\uppsi(-1/2)f\big(h(2^n)x(1)\big).
\end{align*}
Since $\uppsi(-1/2)$ is a nontrivial 4th root of 1, $f$ can only be nonzero on the double cosets $Jh(2^n)J$ and $Jw(2^n)J$.
\end{proof}

In order to determine the structure of $\HH$, two subalgebras will be introduced.  Define $J_1$ to be the open compact group
$$J_1=J\cup Jw_1J=J\cup Jw_1;$$
that is, $J_1$ is the stabilizer of $\C\upphi_0$.

Define $\HH_1$ to be the Hecke subalgebra $\HH(\biinv{J_1}{J};\upchi)$ of functions in $\HH$ supported on $J_1$.  Since $J_1$ intersects the affine Weyl group at $1$ and $w_1$, this is either a one- or two-dimensional subalgebra.  Following the notation from Section \ref{heckealgebras}, let $V_1=J_1\cdot\C\upphi_0$ and $U_1=\ind_J^{J_1}\upchi$ of dimensions 
$$\dim(V_1)=\dim(\C\upphi_0)=1\;\;\text{ and }\;\;\dim(U_1)=[J_1:J]=2.$$  
By Proposition \ref{T_x}, $\HH_1$ is the two-dimensional algebra generated by $\dot\upchi$ and $T_1$, where $T_1$ is supported on $Jw_1J$ and is normalized to act on $V_1$ by
$$\uplambda=\frac{2-1}{1}=1.$$
In other words, $T_1$ satisfies the quadratic relation is $T_1^2=1$.

Define $J_0$ to be the full inverse image of
$$\underline{J}_0=\left\{\mtwo{a}{b/2}{2c}{d}\in\SL_2(\Q_2):a,b,c,d\in\Z_2\right\}.$$
Note that $\underline{J}_0$ and $\SL_2(\Z_2)$ are the two distinct maximal compact subgroups of $\SL_2(\Q_2)$ that contain the Iwahori subgroup.  The group $J_0$ is generated by $x(1/2)$ and $y(2)$ or, equivalently, by $x(1/2)$ and $w_0=w(1/2)$.  Since
\begin{align*}
x(1/2)\upphi_0(y)&=\uppsi(1/2)\upphi_0(y)+\upphi_1(y),\\
x(1/2)\upphi_1(y)&=\upphi_1(y),\\
w(1/2)\upphi_0(y)&\in\C\widehat\upphi_0(y/2)=\C\upphi_0(y/2)=\C\upphi_1(y),\\
w(1/2)\upphi_1(y)&\in\C\widehat\upphi_1(y/2)=\C\upphi_{-1}(y/2)=\C\upphi_0(y),
\end{align*}
the space $V_0=J_0\cdot\C\upphi_0$ is the two-dimensional space $\C\upphi_0\oplus\C\upphi_1$.  Let $\HH_0$ be the subalgebra $\HH(\biinv{J_0}{J};\upchi)$.  Again, as $J_0$ intersects the affine Weyl group only at $1$ and $w_0$, this is a one- or two-dimensional subalgebra.   Let $U_0$ be the induced representation $\ind_J^{J_0}\upchi$.  The dimension of $U_0$ is
$$\dim(U_0)=[J_0:J]=\frac{\vol(\underline{J}_0)}{\vol(\underline{I})}\cdot\frac{\vol(\underline{I})}{\vol(\underline{J})}=3\cdot2=6.$$
As a remark, this may also be computed by observing that the volumes of $\underline{J}_0$ and $\SL_2(\Z_2)$ are equal since they are conjugate, so the index of $\underline{J}$ in $\underline{J}_0$ is 6.  By Proposition \ref{T_x}, $\HH_0$ is the two-dimensional algebra generated by $\dot\upchi$ and $T_0$, where $T_0$ is supported on $Jw_0J$ and is normalized to act on $V_0$ by
$$\uplambda=\frac{6-2}{2}=2.$$
Hence, $T_0$ satisfies the quadratic relation $(T_0-2)(T_0+1)=0$.  

Next, for each $w\in W^\text{aff}$, a nontrivial element $T_w$ supported on $JwJ$ will be defined.  Let $w=w_{i_1}\cdots w_{i_m}$ be a minimal expression for $w$ in terms of the simple generators $w_0$, $w_1$.  This minimal expression is unique since there are no braid relations in $W^\text{aff}$.  Define $T_w=T_{i_1}\cdots T_{i_m}$.  Then,
$$\op{supp}(T_w)=Jw_{i_1}J\cdots Jw_{i_m}J\subset Iw_{i_1}I\cdots Iw_{i_m}I=IwI,$$
the last equality following from the minimality of the expression for $w$.  Since $T_w$ can only be supported on $JW^\text{aff}J$, the support of $T_w$ is exactly $JwJ$.  As $T_0$ and $T_1$ act on $\C\upphi_0$ by 2 and 1, respectively, $T_w$ is nontrivial.

Considering the support of each $T_w$, there are no braid relations between $T_0$ and $T_1$ since there are no braid relations between $w_0$ and $w_1$ in $W^\text{aff}$.

The result of this discussion is the following theorem.

\begin{theorem}
The support of $\HH$ is $JW^\emph{aff}J$ and $\{T_w:w\in W^\emph{aff}\}$ forms a basis for $\HH$ as a vector space.  As an abstract $\C$-algebra, $\HH$ is generated by $T_0$ and $T_1$, subject only to the quadratic relations
$$T_1^2=1\;\;\text{ and }\;\;(T_0-2)(T_0+1)=0.$$
In particular, $\HH$ is isomorphic to the Iwahori-Hecke algebra of $\op{PGL}_2(\Q_2)$.
\end{theorem}

\begin{proof}
The only remark to be made concerns the Hecke algebra of $\op{PGL}_2(\Q_2)$.  Let $G'=\op{PGL}_2(\Q_2)$, $I'$ its Iwahori subgroup and $\HH'$ its Iwahori-Hecke algebra.  Denote by $t_0$ and $t_1$ the characterstic functions of
$$I'\mtwo{0}{1}{2}{0}I'\;\;\text{ and }\;\;I'\mtwo{0}{1}{1}{0}I',$$
respectively.  They generate $\HH'$ as an abstract $\C$-algebra, have no braid relations between them, and satisfy the quadratic relations
$$t_0^2=1\;\;\text{ and }\;\; (t_1-2)(t_1+1)=0,$$
\cf \cite{loke-savin}.  Hence, the isomorphism between $\HH$ and $\HH'$ is given by
$$T_0\leftrightarrow t_1\;\;\text{ and }\;\; T_1\leftrightarrow t_0.$$
\end{proof}

\section{A Minimal Even Type for $\ce\Sp_{2\lowercase{n}}(\Q_2)$}

Throughout this section, $\W$ is a $2n$-dimensional symplectic vector space over $\Q_2$, $\X+\Y$ is a complete polarization of $\W$, $\underline{G}$ is the symplectic group $\Sp(\W)$, $G$ is a two-fold central extension of $\underline{G}$, $\uppsi$ is an additive character of $\Q_2$ of conductor 1, $V$ is the vector space $S(\Y)$, and $(\upomega,V)$ is the Schr\"odinger model of the Weil representation with respect to $\uppsi$.  Recall that the elements $x(b)$, $h(b)$, and $w$ in $G$ are lifts of
$$\mtwo{1}{b}{0}{1},\;\;\;\mtwo{b}{0}{0}{\tpinv{b}},\;\text{ and }\;\mtwo{0}{1}{-1}{0},\;\text{ respectively,}$$
and that the Weil representation of $G$ on $V$ is given by
\begin{align*}
x(b)\upphi(y)&=\uppsi(\tp{y}by)\upphi(y)\\
h(b)\upphi(y)&=\upbeta_b\upphi(\tp{b}y)\\
w\upphi(y)&=\upgamma_1\widehat{\upphi}(y).
\end{align*}

As a central extension of a Chevalley group, $G$ is generated by $x_\upalpha(t)$ for $\upalpha\in\Phi$ and $t\in\Q_2$, and is subject to Steinberg's relations for covering groups, \cf Section \ref{relations}.  The generators $x_\upalpha(t)$ may be expressed in terms of $x(b)$ and $h(b)$ as follows.
\begin{align*}
\text{ If }\upalpha=\uplambda_i-\uplambda_j,&\text{ then }x_\upalpha(t)=h(1+tE_{ij});\\
\text{ if }\upalpha=\uplambda_i+\uplambda_j,&\text{ then }x_\upalpha(t)=x\big(t(E_{ij}+E_{ji})\big);\\
\text{ if }\upalpha=2\uplambda_i,&\text{ then }x_\upalpha(t)=x(tE_{ii}).
\end{align*}
Recall that, for an affine root $\upgamma=\upalpha+m$, the affine root group $\mf{X}_\upgamma$ is
$$\mf{X}_{\upgamma}=\{x_{\upgamma}(t):t\in\Z_2\}=\{x_\upalpha(t):t\in2^m\Z_2\}.$$

Define the elements $w_0,w_1,\dots,w_n$ of $G$ as in Section \ref{chevalleygroup}.  These elements are representatives in $G$ of the simple generators of the affine Weyl group $W^\text{aff}$.  Abusing notation again, the elements of $W^\text{aff}$ and their representatives in $G$ will often be referred to using the same symbols.

Consider now the element $w_{2\uplambda_i}(1)$, for $i=1,\dots,n$.  This element acts, up to scalar, as the partial Fourier transform on the $i$th component of a function $\mathbf{f}$, so $w$ may be taken to be the product $\prod w_{2\uplambda_i(1)}$, which is a representative of the longest element of the Weyl group.  Therefore, if $\upalpha$ is any root, then
$$w^{-1}x_\upalpha(t)w=x_{-\upalpha}(-t).$$

\subsection{Minimal Type}

Let $\uprho:\Sp_{2n}(\Z_2)\to\Sp_{2n}(\F_2)$ be projection modulo 2, and define $B=B_{2n}(\F_2)$ to be the Borel subgroup of $\Sp_{2n}(\F_2)$ whose unipotent radical is generated by the positive root groups.  The Iwahori subgroup of $\Sp_{2n}(\Z_2)\subset\underline{G}$ is 
$$\underline{I}=\uprho^{-1}(B).$$
Let $I$ be the full inverse image in $G$ of $\underline{I}$,  generated by
$$\big\{h_\upalpha(t): \upalpha\in\Phi,t\in\Z_2^\times\big\}\cup\left\{x_\upalpha(t):
\begin{array}{l}
t\in\Z_2\text{ for }\upalpha>0\\
t\in2\Z_2\text{ for }\upalpha<0
\end{array}\right\}.$$
In terms of affine root groups, the second set of generators is
$$\left\{\mf{X}_{\upalpha+m}:
\begin{array}{l}
m\geq0\text{ for }\upalpha>0\\
m\geq1\text{ for }\upalpha<0
\end{array}\right\}.$$

As in the $\ce\SL_2$ case, a subgroup $J\subset I$ is needed in order to construct a minimal type.  Over $\F_2$, the finite split orthogonal group $\op{O}_{2n}(\F_2)$ is defined as the set of linear operators under which a symmetric quadratic form $q$ is invariant.  Take $q$ to be the quadratic form
$$q(x_1,\dots,x_n,y_1,\dots,y_n)=\sum_{i=1}^nx_iy_i$$
in the basis $\{e_1,\dots,e_n,f_1,\dots,f_n\}$.  Associated with this $q$ is the bilinear form 
$$(u,v)=q(u+v)+q(u)+q(v),$$ 
which is clearly symmetric.  In characteristic 2, a symmetric bilinear form is also skew-symmetric, so $\op{O}_{2n}(\F_2)$ may be realized as a subgroup of $\Sp_{2n}(\F_2)$.  Let $B'=B'_{2n}(\F_2)$ be the Borel subgroup of the finite orthogonal group corresponding to the positive roots, so that $B'$ is realized as a subgroup of $B$; the unipotent radical of $B'$ is generated by the positive short root groups, \cf \cite{carter}.  Define
$$\underline{J}=\uprho^{-1}(B'),$$
which is a subgroup of $\underline{I}$.
\begin{center}
\begin{tikzpicture}[scale=.4]
	\node at (0,4) {$\underline{J}$};
	\node at (4,4) {$\underline{I}$};
	\node at (9.5,4) {$\Sp_{2n}(\Z_2)$};
	\node at (0,0) {$B'$};
	\node at (4,0) {$B$};
	\node at (9.5,0) {$\Sp_{2n}(\F_2)$};
	\draw[->] (1,4) --(3,4);
	\draw[->] (5,4) --(7,4);
	\draw[->] (1,0) --(3,0);
	\draw[->] (5,0) --(7,0);
	\draw[->] (0,3) --(0,1);
	\draw[->] (4,3) --(4,1);
	\draw[->] (9.5,3) --(9.5,1);
	\node at (-.5,2) {$\uprho$};
	\node at (3.5,2) {$\uprho$};
	\node at (9,2) {$\uprho$};
\end{tikzpicture}
\end{center}
Let $J$ be the full inverse image in $G$ of $\underline{J}$, generated by
$$\big\{h_\upalpha(t): t\in\Z_2^\times\big\}\cup\left\{x_\upalpha(t):
\begin{array}{l}
t\in\Z_2\text{ for }\upalpha>0,\text{ short}\\
t\in2\Z_2\text{ for }\upalpha>0,\text{ long}\\
t\in2\Z_2\text{ for }\upalpha<0
\end{array}\right\}$$
In terms of affine root groups, the second set of generators is
$$\left\{\mf{X}_{\upalpha+m}:
\begin{array}{l}
m\geq0\text{ for }\upalpha>0,\text{ short}\\
m\geq1\text{ for }\upalpha>0,\text{ long}\\
m\geq1\text{ for }\upalpha<0
\end{array}\right\}.$$
The group $J$ is a normal subgroup of $I$, and the quotients $\rQ{I}{J}$ and $\lQ{J}{I}$ have coset representatives of the form
$$x=\prod_{\upalpha\in S}x_\upalpha(1),$$
where $S$ is any subset of the positive long roots.  The order in which the product is taken is of no consequence, as the root groups $\mf{X}_\upalpha$ and $\mf{X}_\upbeta$ commute with each other for any long roots $\upalpha\neq-\upbeta$.

Consider the even function 
$$\mathbf{f}=\upphi_0\otimes\cdots\otimes\upphi_0,$$
which is the characteristic function of the $\Z_2$-span of the basis $\{f_1,\dots,f_n\}$ of $\Y$; its Fourier transform is $\widehat{\mathbf{f}}=\mathbf{f}$.  The stabilizer $\op{Stab}_G(\C\mathbf{f})$ of the line $\C\mathbf{f}$ will be computed.  First, the Weyl group elements $w_\upalpha(1)$ will be discussed.  Since $w$ preserves $\C\mathbf{f}$ and conjugation by $w$ takes $w_\upalpha(1)$ to $w_{-\upalpha}(1)$, only the positive roots require consideration.
\begin{enumerate}[ \mbox{ } 1.]
\item If $\upalpha$ is the positive long root $2\uplambda_i$, then $w_\upalpha(1)$ essentially acts on $\mathbf{f}$ by partial Fourier transform on the $i$th component.  Hence, $w_\upalpha(1)$ preserves $\C\mathbf{f}$.
\item If $\upalpha$ is the positive short root $\uplambda_i-\uplambda_j$, then $w_\upalpha(1)=h(b)$ where $b=1+E_{ij}-E_{ji}-E_{ii}-E_{jj}$.  Hence, $w_\upalpha(1)$ essentially acts on $\mathbf{f}$ by switching the $i$th and $j$th coordinates, which clearly preserves the line $\C\mathbf{f}$.
\item If $\upalpha$ is the positive short root $\uplambda_i+\uplambda_j$, then $\upalpha$ may be obtained from $\uplambda_i-\uplambda_j$ via conjugation by $2\uplambda_j$.  Therefore, $w_\upalpha(1)$ also preserves the line $\C\mathbf{f}$.
\end{enumerate}

Next, the elements $h_\upalpha(t)$ will be considered.  If the stabilizer group contains $h_\upalpha(t)$, it also must contain $h_\upalpha(t^{-1})$.  Each $h_\upalpha(t)$ is $h(b)$ with $b=(b_1,\dots,b_n)$ a diagonal matrix.  The action of $h_\upalpha(t)$ on $\mathbf{f}(y)$ is essentially $\mathbf{f}(\tp{b}y)=\mathbf{f}(b_1y_1,\dots,b_ny_n)$.  Since $1$, $t$, and $t^{-1}$ are exactly the entries of $b$ for $h_\upalpha(t)$ and $h_\upalpha(t^{-1})$, 
$$h_\upalpha(t)\in\op{Stab}_G(\C\mathbf{f})\iff t\in\Z_2^\times.$$
Note that the same condition holds for the elements $w_\upalpha(t)$, since $w_\upalpha(t)=h_\upalpha(t)w_\upalpha(1)$.

Lastly, the elements $x_\upalpha(t)$ will be discussed; it suffices to consider the elements $x_\upalpha(2^m)=x_{\upalpha+m}(1)$.  Since $w^{-1}\mf{X}_\upalpha w=\mf{X}_{-\upalpha}$, the condition on $m$ will be the same for both $\upalpha$ and $-\upalpha$.  Therefore, only the positive affine root groups will be considered.
\begin{enumerate}[ \mbox{ } 1.]
\item If $\upalpha=\uplambda_i-\uplambda_j$ for $i<j$, then $x_{\upalpha+m}(1)=h(b)$ with $b=1+2^mE_{ij}$.  Since $\tp{b}$ acts on $y$ by replacing $y_j$ with $y_j+2^my_i$,
$$\mf{X}_{\upalpha+m}\subset\op{Stab}_G(\C\mathbf{f})\iff m\geq0.$$
\item If $\upalpha=\uplambda_i+\uplambda_j$, then $\upalpha$ is conjugate to $\upbeta=\uplambda_j-\uplambda_i$ by $2\uplambda_i$.  Therefore, $x_\upalpha(t)=w_{2\uplambda_i}(1)^{-1}x_\upbeta(\pm t)w_{2\uplambda_i}(1)$, so
$$\mf{X}_{\upalpha+m}\subset\op{Stab}_G(\C\mathbf{f})\iff m\geq0.$$
\item If $\upalpha=2\uplambda_i$, then $x_{\upalpha+m}(1)=x(b)$ with $b=2^mE_{ii}$.  Since $\tp{y}by=2^my_i^2$, $x_\upalpha(t)$ acts on $\mathbf{f}$ by $\uppsi(2^my_i^2)$, which is independent of $y$ exactly for $m\geq1$.  Hence,
$$\mf{X}_{\upalpha+m}\subset\op{Stab}_G(\C\mathbf{f})\iff m\geq1.$$
\end{enumerate}
In conclusion, 
$$J=I\cap\op{Stab}_G(\C\mathbf{f}),$$
so the Weil representation restricted to $J$ acts on the line $\C\mathbf{f}$.  Define $\bar\upchi$ to be the character of $J$ given by $\upomega(j)\mathbf{f}=\bar\upchi(j)\mathbf{f}$.  This character acts on the nontrivial subspace
$$V^{J,\bar\upchi}=\{\upphi\in V:\upomega(j)\upphi=\bar\upchi(j)\upphi\text{ for }j\in J\}\supset\C\mathbf{f}.$$

\begin{prop}\label{Sp_2n-chibar}
The character $\bar\upchi$ satisfies the following properties.
\begin{enumerate}[ \mbox{ } 1.]
\item $\bar\upchi\big(x_\upalpha(t)\big)=1$, for $t\in2\Z_2$ and $\upalpha$ postive, long.
\item $\bar\upchi\big(x_\upalpha(-1)x_{-\upalpha}(t)x_\upalpha(1)\big)=\uppsi(-t/4),$ for $t\in2\Z_2$ and $\upalpha$ positive, long.
\item $V^{J,\bar\upchi}=\C\mathbf{f}$.
\end{enumerate}
\end{prop}

\begin{proof}
The first part follows from the above computations.  By the same computations, it is clear that the action of $x_\upalpha(-1)x_{-\upalpha}(t)x_\upalpha(1)$ on the appropriate component of $S(\Y)$ is the same as the action of $x(-1)y(t)x(1)$ from the $\ce\SL_2$ setting.  Hence, the second part follows from Proposition \ref{chibar}.

For the third part, recall that a function $\upphi\in V^{J,\bar\upchi}$ is the tensor product of functions in $S(\Q_2)$.  By Proposition \ref{chibar}, each piece of  $\upphi$ must be a multiple of $\upphi_0$, hence the tensor product must be a multiple of $\mathbf{f}$.
\end{proof}

\subsection{Hecke Algebra}

Recalling again Section \ref{heckealgebras}, the Hecke algebra for this type  is $\HH=\HH(\biinv{G}{J};\upchi).$  A function $f\in\HH$ is determined by its value on double coset representatives of $\dQ{J}{G}{J}$.  By \cite{iwahori-matsumoto}, $\underline{G}=\underline{I}W^\text{aff}\underline{I}$, so a typical $J$-double coset is of the form $Jx_1wx_2J$, where $w$ is a representative an element of the affine Weyl group, and $x_1, x_2$ are products of the form $\prod_{\upalpha\in S}x_\upalpha(1)$ for a subset $S$ of the positive long roots.

Some of these representatives are redundant and may be eliminated as follows.  Let $x_1wx_2$ be a representive of a typical $J$-double coset.  Suppose that $x_\upalpha(1)$ occurs in $x_1$ and that $w^{-1}x_\upalpha(1)w=x_\upbeta(t)$ for some root $\upbeta$ and some $t\in2\Z_2$.  It may be assumed that $x_1=x'_1x_\upalpha(1)$.  If $x_{-\upbeta}(1)$ does not occur in $x_2$, then 
$$Jx'_1x_\upalpha(1)wx_2J=Jx'_1wx_2x_\upbeta(t)J=Jx'_1wx_2J.$$
If $x_{-\upbeta}(1)$ does occur in $x_2$, say $x_2=x'_2x_{-\upbeta}(1)$, then
\begin{align*}
Jx'_1x_\upalpha(1)wx_2J
&=Jx'_1wx'_2x_\upbeta(t)x_{-\upbeta}(1)J\\
&=Jx'_1wx'_2x_{-\upbeta}(1)\big(x_{-\upbeta}(-1)x_\upbeta(t)x_{-\upbeta}(1)\big)J\\
&=Jx'_1wx_2J,
\end{align*}
with the last equality following from the fact that $J$ is normal in $I$.  In either case, if $t\in2\Z_2$, then $x_\upalpha(1)$ may be moved across and absorbed into $J$.

Supposing now that $x_\upalpha(1)$ occurs as before with $t\notin2\Z_2$, suppose further that $x_{\upbeta}(1)$ occurs in $x_2$.  Since $wx_{\upbeta}(1)w^{-1}=x_\upalpha(1/t)$, a similar argument to that above permits the assumption that $1/t\notin2\Z_2$; that is, it may be assumed that $t\in\Z_2^\times$, or that $1+t\in2\Z_2$.  Writing $x_1=x'_1x_\upalpha(1)$ and $x_2=x'_2x_\upbeta(1)$, one has
$$Jx'_1x_\upalpha(1)wx'_2x_\upbeta(1)J=Jx'_1wx'_2x_\upbeta(1+t)J=Jx'_1wx'_2J.$$
Therefore, it may be assumed that an $x_\upalpha(1)$ occurring in $x_1$ precludes the appearance of the corresponding $x_\upbeta(1)$ in $x_2$.

In summary, if $x_\upalpha(1)$ occurs in $x_1$ with $w^{-1}x_\upalpha(1)w=x_\upbeta(t)$, it may be assumed that $t\notin2\Z_2$ and that $x_\upbeta(1)$ does not occur in $x_2$.  

\begin{lemma}\label{support1}
The support of $\HH$ is contained in $JW^{\op{aff}}J$.
\end{lemma}

\begin{proof}
Let $x_1wx_2$ be a typical double coset representative and let $x_\upalpha(1)$ be the first term of $x_1$.  For $w^{-1}x_\upalpha(1)w=x_\upbeta(t)$, it is assumed that $t\notin2\Z_2$ and that $x_\upbeta(1)$ does not occur in $x_2$.  Using Proposition \ref{Sp_2n-chibar} and the relations from Section \ref{relations}, for $f\in\HH$,
\begin{align*}
f(x_1wx_2)
&=f(x_1wx_2)\upchi\big(x_{-\upbeta}(2/t)\big)\\
&=f\big(x_1wx_2x_{-\upbeta}(2/t)\big)\\
&=f\big(x_1wx_{-\upbeta}(2/t)x_2\big)\\
&=f\big(x_1x_{-\upalpha}(2)wx_2\big)\\
&=f\big(x_\upalpha(1)x_{-\upalpha}(2)x_\upalpha(-1)x_1wx_2\big)\\
&=\upchi\big(x_\upalpha(1)x_{-\upalpha}(2)x_\upalpha(-1)\big)f(x_1wx_2)\\
&=\uppsi(-1/2)f(x_1wx_2).
\end{align*}
Since $\uppsi(-1/2)$ is a nontrivial 4th root of 1, it must be that $f(x_1wx_2)=0$.  By induction, if $x_1$ is nontrivial, then $f(x_1wx_2)=0$.  A symmetrical calculation shows that  if $x_2$ is nontrivial, then $f(wx_2)=0$.  Therefore, the support of $\HH$ is at most the set of $J$-double cosets of the affine Weyl group.
\end{proof}

To study the structure of $\HH$, the first step is to find nontrivial elements $T_0,\dots,T_n$ supported on the double cosets $Jw_0J,\dots,Jw_nJ$ and to compute their quadratic relations.

The $i=n$ case will be discussed first.  Define $J_n=J\cup Jw_nJ=J\cup Jw_n$.  Since $w_n$ normalizes $J$ and the action of $w_n$ on the last component of $\mathbf{f}$ mirrors the action of $w_1$ from the $\ce\SL_2$ setting, the situation is the same as in the previous section.  In particular, there exists an element $T_n$ of $\HH$, supported on $Jw_nJ$, which satisfies the quadratic relation $T_n^2=1$.

Next, consider $i=1,\dots,n-1$.  Define $J_i=J\cup Jw_iJ$.  For the finite group $\op{O}_{2n}(\F_2)\subset\Sp_{2n}(\F_2)$, let $P'_i(\F_2)$ be the parabolic subgroup of $\op{O}_{2n}(\F_2)$ whose unipotent radical is obtained by excluding the short simple root $\upalpha_i$.  Then, 
$$\underline{J}_i=\uprho^{-1}\big(P'_i(\F_2)\big),$$
where $\uprho$ is the projection $\Sp_{2n}(\Z_2)\to\Sp_{2n}(\F_2)$ and $\underline{J}_i$ is the linear projection of $J_i$.  

Define the Hecke subalgebra $\HH_i=\HH(\biinv{J_i}{J};\upchi)$ of functions supported on $J_i$.  Each $J_i$ intersects the affine Weyl group at $1$ and $w_i$, so $\HH_i$ is either one- or two-dimensional.  Define $V_i=J_i\cdot\C\mathbf{f}$ and $U_i=\ind_J^{J_i}\upchi$.  Since $w_i$ acts by switching the $i$ and $i+1$ coordinates, $V_i$ is the one-dimensional space $\C\mathbf{f}$.  The dimension of $U_i$ is
$$\dim(U_i)=[J_i:J]=[P'_i(\F_2):B']=3.$$
By Proposition \ref{T_x}, $\HH_i$ is the two-dimensional algebra generated by $\dot\upchi$ and $T_i$, where $T_i$ is supported on $Jw_iJ$ and is normalized to act on $V_i$ by
$$\uplambda=\frac{3-1}{1}=2.$$
Hence, $T_i$ satisfies the quadratic relation
$$(T_i-2)(T_i+1)=0.$$

Finally, the $i=0$ case will be described.  The definition of $J_0$ is somewhat of a delicate issue, so some facts about the standard apartment of $\Sp_{2n}$ will be recalled, \cf \cite{iwahori-matsumoto}, \cite{humphreys-coxeter}.

The affine Weyl group acts on the standard apartment in $\h_\R$ via affine reflections.  Take $\{a_1,\dots,a_n\}$ to be the basis of $\h_\R$ dual to $\{\uplambda_1,\dots,\uplambda_n\}$.  For $\upgamma\in W^\text{aff}$, let $P_\upgamma$ be the hyperplane fixed under the reflection associated to $\upgamma$ and let $H_\upgamma$ be the open half-plane on which $\upgamma$ acts positively.  The fundamental chamber of the apartment is
$$C=\bigcap_{i=0}^n H_{\upalpha_i},$$
and the Iwahori subgroup $\underline{I}$ of $\Sp_{2n}(\Z_2)$ is generated by the diagonal elements $\{h_\upalpha(t):\upalpha\in\Phi,t\in\Z_2^\times\}$ and the affine root groups
$$\{\mf{X}_\upgamma:\upgamma(C)>0\}.$$
The origin $0=(0,\dots,0)$ of $\h_\R$ corresponds to the maximal compact $\underline{K}=\Sp_{2n}(\Z_2)$ in the sense that $\underline{K}$ is generated by
$$\{\mf{X}_\upgamma:\upgamma(0)\geq0\}.$$
Consider the point $z=(1/2,0,\dots,0)\in\h_\R$ and let $\underline{K}'$ be the open compact subgroup generated by
$$\{\mf{X}_\upgamma:\upgamma(z)\geq0\}.$$
The group $\underline{K}'$ has a quotient $\SL_2(\F_2)\times \Sp_{2n-2}(\F_2)$.  If $\uprho'$ is the quotient map and $B'_{2n-2}(\F_2)$ is the Borel subgroup of $\op{O}_{2n-2}(\F_2)\subset\Sp_{2n-2}(\F_2)$, then $\underline{J}_0$ is defined as
$$\underline{J}_0=(\uprho')^{-1}\big(\SL_2(\F_2)\times B'_{2n-2}(\F_2)\big),$$
and $J_0$ is taken to be the full inverse image in $G$ of $\underline{J}_0$.  

Let $\HH_0$ be the Hecke subalgebra $\HH(J_0/\!\!/J;\upchi)$ of functions supported on $J_0$.  As $J_0$ intersects $W^\text{aff}$ at $1$ and $w_0$, $\HH_0$ is either one- or two-dimensional.

Define $V_0=J_0\cdot\C\mathbf{f}$.  Since $w_0=w_{2\uplambda_1}(1/2)$, and $x_{2\uplambda_1}(1/2)$ act on the first component of $V$ in the same way that $w(1/2)$ and $x(1/2)$ acted in the $\ce\SL_2$ case, the space $V_0$ is at least
$$V'_0=\big(\C\upphi_0\oplus\C\upphi_1\big)\otimes\upphi_0\otimes\cdots\otimes\upphi_0.$$
The claim is that this is a type for $J_0$, which gives $V_0=V'_0$.  What remains to be verified is that $\mf{X}_\upgamma$ preserves $V'_0$ for $\upgamma(z)\geq0$.  Only the short root groups involving $\uplambda_1$ need be considered.
\begin{enumerate}[ \mbox{ } 1.]
\item Let $\upalpha=\uplambda_1-\uplambda_j$, so $(\upalpha+m)(z)\geq0$ for $m\geq0$.  For such $\upalpha+m$, $x_{\upalpha+m}(1)=h(b)$ with $b=1+2^mE_{1j}$.  Recall that $\tp{b}$ acts on $y$ by replacing $y_j$ with $y_j+2^my_1$.  Since the $j$th component of a function in $V'_0$ is essentially $\upphi_0$, $V'_0$ is invariant under such a geometric action.
\item Let $\upalpha=-\uplambda_1+\uplambda_j$, so $(\upalpha+m)(z)\geq0$ for $m\geq1$.  In this setting, $x_{\upalpha+m}(1)=h(b)$ with $b=1+2^mE_{j1}$, and $\tp{b}$ acts on $y$ by replacing $y_1$ with $y_1+2^my_j\in y_1+2\Z_2$.  The first component of a function in $V'_0$ is a function on $\rQ{\Z_2}{2\Z_2}$, so $V'_0$ is invariant under the action $y\mapsto y+2\Z_2$.
\item Note that $\mf{X}_{\uplambda_1+\uplambda_j}$ and $\mf{X}_{-\uplambda_1-\uplambda_j}$ are respectively conjugate to $\mf{X}_{\uplambda_1-\uplambda_j}$ and $\mf{X}_{-\uplambda_1+\uplambda_j}$ via $w_{2\uplambda_j}(1)$ and $w_{2\uplambda_j}(1)^{-1}$.  Since $w_{2\uplambda_j}(1)$ preserves $V'_0$, so do the appropriate $\mf{X}_{\uplambda_1+\uplambda_j+m}$ and $\mf{X}_{-\uplambda_1-\uplambda_j-m}$.
\end{enumerate}
Hence, $V_0$ is a 2-dimensional type for $J_0$.

Define $U_0$ to be the induced representation $\ind_J^{J_0}\upchi$.  Under $\uprho'$, the Iwahori group $\underline{I}$ is the inverse image of $B_2(\F_2)\times B_{2n-2}(\F_2)$, and $\underline{J_0}$ is the inverse image of $\SL_2(\F_2)\times B'_{2n-2}(\F_2)$.  Since
$$[\SL_2(\F_2):B_2(\F_2)]=3\;\;,\;\;[B_{2n-2}:B'_{2n-2}]=2^{n-1},$$
and the index of $J$ in $I$ is $2^n$, the dimension of $U_0$ is
$$\dim(U_0)=[J_0:J]=\frac{\vol(\underline{J}_0)}{\vol(\underline{I})}\cdot\frac{\vol(\underline{I})}{\vol(\underline{J})}=(3\cdot2^{-n+1})\cdot (2^n)=6.$$
By Proposition \ref{T_x}, $\HH_0$ is two-dimensional and is generated by $\dot\upchi$ and $T_0$, where $T_0$ is supported on $Jw_0J$ and is normalized to act on $V_0$ by
$$\uplambda=\frac{6-2}{2}=2.$$
In other words, $T_0$ satisfies the quadratic relation
$$(T_0-2)(T_0+1)=0.$$

The next step is find a nontrivial element $T_w$ of $\HH$ supported on $JwJ$ for each $w\in W^\text{aff}$.  To do this, let $w=w_{i_1}\cdots w_{i_m}$ be a minimal expression in terms of $\{w_0,w_1,\dots,w_n\}$, and define
$$T_w=T_{i_1}\cdots T_{i_m}.$$
It remains to be verified that this definition does not depend on the choice of minimal expression, but first, the support of $T_w$ will be discussed.  The support of $T_w$ satisfies
$$\op{supp}(T_w)=Jw_{i_1}J\cdots Jw_{i_m}J\subset Iw_{i_1}I\cdots Iw_{i_m}I=IwI.$$
The last equality follows from the fact that the expression for $w$ is minimal.  Each $I$-double coset is a finite union of $J$-double cosets, but since the support of $\HH$ is contained in $JW^\text{aff}J$, the support of $T_w$ must be exactly $JwJ$.  

Each of $T_0,\dots,T_{n-1}$ acts on $\C\mathbf{f}$ by 2 and $T_n$ acts by 1.  Since the braid relations for $w_n$ are
$$w_nw_{n-1}w_nw_{n-1}=w_{n-1}w_nw_{n-1}w_n$$
and
$$w_nw_i=w_iw_n\text{ for }i<n-1,$$
the number $k(w)$ of $w_n$ occurring in a minimal expression is independent of the choice of expression.  Therefore, $T_w$ acts on $\C\mathbf{f}$ by $2^{m-k(w)}$.  Given that $T_w$ is determined by its support and its action on $\C\mathbf{f}$, each of these $T_w$ is well-defined.  This discussion also gives that the braid relations of $T_0,\dots,T_n$ are the same as the braid relations of $w_0,\dots,w_n$.  

The following theorem summarizes these results.

\begin{theorem}
The support of $\HH$ is $JW^\emph{aff}J$ and $\{T_w:w\in W^\emph{aff}\}$ forms a basis for $\HH$ as a vector space.  As an abstract algebra, $\HH$ is generated by $T_0,\dots,T_n$, subject to the quadratic relations
$$T_n^2=1\;\;\text{ and }\;\;
(T_i-2)(T_i+1)=0\;\;\text{ for }\;\;i<n,$$
and the braid relations given by the following Coxeter diagram.
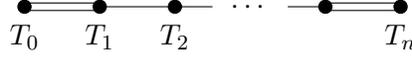
\begin{center}
\begin{tikzpicture}[scale=.5]
	\draw (0,-.1) --(2,-.1);
	\draw (0,.1) --(2,.1);
	\draw (2,0) --(5,0);
	\draw (7,0) --(8,0);
	\draw (8,-.1) --(10,-.1);
	\draw (8,.1) --(10,.1);
	\draw[fill] (0,0) circle(5pt);
	\draw[fill] (2,0) circle(5pt);
	\draw[fill] (4,0) circle(5pt);
	\draw[fill] (8,0) circle(5pt);
	\draw[fill] (10,0) circle(5pt);
	\node at (0,-.8) {$T_0$};
	\node at (2,-.8) {$T_1$};
	\node at (4,-.8) {$T_2$};
	\node at (10,-.8) {$T_n$};
	\node at (6,0) {$\cdots$};
\end{tikzpicture}
\end{center}
In particular, $\HH$ is isomorphic to the Iwahori-Hecke algebra of $\op{SO}_{2n+1}(\Q_2)$.
\end{theorem}

\begin{proof}
The only detail left is the isomorphism between the two algebras.  Let $\HH'$ be the Iwahori-Hecke algebra of $\op{SO}_{2n+1}(\Q_2)$.  It is generated by $t_0,\dots,t_n$ and $\uptau$, where $\uptau$ corresponds to the involution of the extended diagram of type $B_n$ that switches the $t_0$ and $t_1$ vertices.
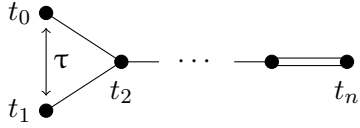
\begin{center}
\begin{tikzpicture}[scale=.5]
	\draw (0,-1.3) --(2,0);
	\draw (0,1.3) --(2,0);
	\draw (2,0) --(3,0);
	\draw (5,0) --(6,0);
	\draw (6,-.1) --(8,-.1);
	\draw (6,.1) --(8,.1);
	\draw[<->] (0,.9) --(0,-.9);
	\draw[fill] (0,1.3) circle(5pt);
	\draw[fill] (0,-1.3) circle(5pt);
	\draw[fill] (2,0) circle(5pt);
	\draw[fill] (6,0) circle(5pt);
	\draw[fill] (8,0) circle(5pt);
	\node at (-.7,1.3) {$t_0$};
	\node at (-.7,-1.3) {$t_1$};
	\node at (2,-.8) {$t_2$};
	\node at (8,-.8) {$t_n$};
	\node at (4,0) {$\cdots$};
	\node at (.4,0) {$\uptau$};
\end{tikzpicture}
\end{center}
Since $\uptau t_1\uptau=t_0$, the generator $t_0$ is unnecessary to define $\HH'$ abstractly.  The braid relation between $\uptau$ and $t_1$ is
$$\uptau t_1\uptau t_1=t_0t_1=t_1t_0=t_1\uptau t_1\uptau,$$
so $\HH'$ has the same braid relations as $\HH$.  The quadratic relations are
$$\uptau^2=1\;\;\text{ and }\;\;(t_i-2)(t_i+1)=0,$$
and the isomorphism from $\HH'$ to $\HH$ is given by
$$\uptau\mapsto T_n,\hskip15pt\text{ and }\hskip15pt t_i\mapsto T_{n-i}\hskip10pt\text{ for }i=1,\dots,n.$$
\end{proof}

\begin{corollary}
If the Haar measure on $SO_{2n+1}(\Q_2)$ is normalized so that the Iwahori subgroup has volume 1 and the Haar measure on $\ce\Sp_{2n}(\Q_2)$ is normalized  so that $J$ has volume 1, then the Plancherel measures induced by $\HH$ and $\HH'$ coincide.
\end{corollary}

\begin{proof}
By the transfer theorem in \cite{bushnell-henniart-kutzko}, it suffices to prove that the operations $\tr$ and $\ast$ are preserved under $\HH\cong\HH'$.  Since $\HH$ and $\HH'$ are supported on their respective Weyl groups,
$$\tr(T_w)=\begin{cases}1&\text{ if }w=1\\0&\text{ if }w\neq1\end{cases}\;\;\text{ and }\;\;\tr(t_w)=\begin{cases}1&\text{ if }w=1\\0&\text{ if }w\neq1,\end{cases}$$
so the trace operation is preserved by the isomorphism.

In the algebra $\HH'$ for the split orthogonal group, the element $t_w$ is the characteristic function of $IwI$, where $I$ is the Iwahori subgroup and $w$ is an element of the Weyl group.  Since $IwI=Iw^{-1}I$, the $\ast$-operation on $\HH'$ satisfies $t_i^\ast=t_i$, and hence, $t_w^\ast=t_{w^{-1}}$.  

It remains to show the same in $\HH$: that $T_i^\ast=T_i$, and hence $T_w^\ast=T_{w^{-1}}$.  Since $T_i^\ast$ is supported on $Jw_i^{-1}J=Jw_iJ$, it is a multiple of $T_i$, say $T_i^\ast=cT_i$.  Considering the representation $\C\upphi_0$ of $\HH$, where $T_i$ acts by $\uplambda_i$,
$$\uplambda_i||\upphi_0||^2=\langle T_i\upphi_0,\upphi_0\rangle=\langle \upphi_0,T_i^\ast\upphi_0\rangle=c\uplambda_i||\upphi_0||^2,$$
so this multiple is $c=1$.
\end{proof}

As a remark, this coincidence of Plancherel measures on the appropriate Bernstein components of the metaplectic and split orthogonal groups also holds in the case of residual characteristic $p\neq2$ considered by Gan and Savin in \cite{gan-savin}.  The proof that their isomorphism of Hecke algebras preserves the two operations $\ast$ and $\tr$ is the same as above.


\vskip20pt




\end{document}